\documentclass[11pt]{amsart}

\usepackage[T1]{fontenc}
\usepackage[utf8]{inputenc}
\usepackage[english]{babel}
\usepackage{hyperref}
\usepackage{framed}
\usepackage{xcolor}
\definecolor{shadecolor}{RGB}{220,220,220}
\usepackage{float}
\usepackage[a4paper]{geometry}
\usepackage{indentfirst}

\usepackage{dsfont}
\usepackage[pdftex]{graphicx}

\usepackage{amsmath}
\allowdisplaybreaks  % para que pueda partir fórmulas que ocupan más de una línea, necesita el paquete anterior
\usepackage{amssymb} % para cargar algunos símbolos como \blacksquare y \square
\usepackage{amsfonts} % para cargar algunas fuentes en estilo matemático
\usepackage{enumerate}
\newtheorem{theorem}{Theorem}[section]
\newtheorem{proposition}[theorem]{Proposition}
\newtheorem{definition}[theorem]{Definition}
\newtheorem{lemma}[theorem]{Lemma}
\newtheorem{corollary}[theorem]{Corollary}
\newtheorem{example}[theorem]{Example}

\newtheorem{remark}[theorem]{Remark}

\newcommand{\R}{\mathbb{R}}

\newcommand{\N}{\mathbb{N}}

  \newcommand{\Lip}{\operatorname{Lip}}
\newcommand{\lip}{\operatorname{lip}}
\newcommand{\LIP}{\operatorname{LIP}}
\newcommand{\SLip}{\operatorname{SLip}}

\renewenvironment{proof}{\emph{Proof.}} {\quad \hfill $\blacksquare$ \newline} % para que aparezca un cuadrado negro al acabar la demostración

\usepackage{makecell}

\usepackage{fancyhdr}                     % para definir distintos tipos de cabeceras y pies de página

\begin{document}

\title{{Pointwise semi-Lipschitz functions and Banach-Stone theorems}}

\author[F. Venegas M.]{Francisco Venegas Mart\'inez$^{1}$}
\address{$^{1}$ Universidad de O'Higgins. Instituto de Ciencias de la Ingenier\'ia. \\ Rancagua, O'Higgins, Chile.}
\email{francisco.venegas@uoh.cl}

\author[E. Durand-Cartagena]{Estibalitz Durand-Cartagena$^{2}$}
\address{$^{2}$ Depto. de Matem\'atica Aplicada, ETSI Industriales, UNED\\
28040 Madrid, Spain.} 
\email{edurand@ind.uned.es}

\author[J. {\'A}. Jaramillo]{Jes{\'u}s {\'A}. Jaramillo$^{3}$}
\address{$^{3}$ Depto. de An\'alisis Ma\-te\-m\'a\-ti\-co y Matem\'atica Aplicada, UCM\\
28040 Madrid, Spain.}
\email{jaramil@mat.ucm.es}

\thanks{The research for this work was conducted while the first author was visiting the Department of Mathematical Analysis and Applied Mathematics at UCM during the spring and summer semesters of 2022. J.{\'A}.J. and E.D-C. are partially supported by grant PID2022-138758NB-I00 (Spain).}


\begin{abstract}
We study the fundamental properties of pointwise semi-Lipschitz functions between asymmetric spaces, which are the natural asymmetric counterpart of pointwise Lipschitz functions. We also study the influence that partial symmetries of a given space may have on the behavior of pointwise semi-Lipschitz functions defined on it. Furthermore, we are interested in characterizing the pointwise semi-Lipschitz structure of an asymmetric space in terms of real-valued pointwise semi-Lipschitz functions defined on it. By using two algebras of functions naturally associated to our spaces of pointwise real-valued semi-Lipschitz functions, we are able to provide two Banach-Stone type results in this context. In fact, these results are obtained as consequences of a general Banach-Stone type theorem of topological nature, stated for abstract functional spaces, which is quite flexible and can be applied to many spaces of continuous functions over metric and asymmetric spaces.
\end{abstract}
\maketitle
\tableofcontents
\newpage
During the last years there has been an increasing interest in exploring asymmetric structures in different contexts, and in particular in studying the behavior and properties of spaces endowed with an {\em asymmetric distance}. By this we mean a kind of distance function which does not necessarily satisfies the symmetric property of a usual distance. These spaces are often called {\em asymmetric} or {\em irreversible} metric spaces. We refer to \cite{DJV}, \cite{KZ} and \cite{OZ} as a sample of recent contributions in this line.

In this paper, we focus on the so-called {\em pointwise semi-Lipschitz functions} between asymmetric spaces, which are the natural counterpart in the asymmetric setting of the pointwise Lipschitz functions studied in \cite{DJ} in a metric context. In the real-valued case, pointwise semi-Lipschitz functions encompass also other previously known asymmetric notions, such as functions with {\em finite slope}, which were introduced in \cite{GMT} and have been widely used since then (see e.g. \cite{AGS} or \cite{Io}, and references therein). We also consider the space of functions with {\em uniformly bounded}  pointwise semi-Lipschitz constant, which has a richer structure. 
In the first section (in particular, subsections 1.1 and 1.2), we will illustrate through various examples the similarities and differences between the different functional spaces involved in both the symmetric and asymmetric cases. Additionally, we will provide specific conditions under which these spaces coincide. (Corollary \ref{quasiconvex}). Furthermore, we study the influence that  partial symmetries of a given space may have on the behavior of pointwise semi-Lipschitz functions defined on it. We consider three types of partial symmetry: pointwise, local and global. We compare them and use them to derive useful properties of various classes of pointwise semi-Lipschitz functions concerning stability when composed with real-valued functions or functional separation of sets. All this is done along Section 1.3. 

We are also interested in characterizing the pointwise semi-Lipschitz structure of an asymmetric space $X$ in terms of real-valued pointwise semi-Lipschitz functions defined on $X$. In this way we make a connection with the so-called Banach-Stone-type theorems. This a long and fruitful line of research which, starting with the classical Banach-Stone Theorem, looks for the characterization of topological, metric, smooth, or other kind of structure of a given space $X$, in terms of the algebraic or topological-algebraic structure of a suitable space of real-valued continuous functions on $X$. We refer to \cite{GJEx} for more information about this subject. In our case, this represents a major problem, since our spaces of  pointwise semi-Lipschitz functions do not have a nice algebraic structure. In fact, due to asymmetry, these spaces are not even linear. We are nevertheless able to present in Section 1.4 two Banach-Stone type results, by using two algebras of functions naturally associated to our spaces of pointwise semi-Lipschitz functions (see Theorem \ref{csl-BS} and Theorem \ref{dsl-BS}.) The complete proof of both results is postponed until Section 2, where we obtain a general Banach-Stone type theorem of topological nature, stated for abstract functional spaces, which is quite flexible and can be applied to many spaces of continuous functions over metric and asymmetric spaces, provided some general hypothesis are met (see Theorem \ref{BSt}.) This  result is much inspired by the main result of \cite{V}, where a Banach-Stone type result is given for non-reversible Finsler manifolds using a suitable space of smooth semi-Lipschitz 
functions. Many ideas used in \cite{V} do not rely on specific properties of Finsler manifolds and semi-Lipschitz functions. In fact, most of the 
key notions used in the aforementioned result, such as the definition of the structure space, the weak-star topology of the dual of an asymmetric normed space and the embedding into the structure space, make sense for a general {\em extended asymmetric normed algebra} $\mathcal{A}$ of functions over $X$. Our Theorem \ref{BSt} makes use of these notions, so we start Section 2 by recalling the 
required preliminaries of this general setting. After all this, and as a consequence of Theorem \ref{BSt}, we also deduce a general Banach-Stone type theorem in the Lipschitz setting (Theorem \ref{BSlip}), with a stronger conclusion,  which needs additional requirements on the function spaces to be 
used. Finally, we present an intermediate version in the pointwise Lipschitz case (Theorem \ref{BSpoint}), which can be applied to spaces of functions with bounded metric slope.

\section{Lipschitz analysis on quasi-metric spaces}\label{subsecpointwise}

\subsection{Metric slopes vs pointwise Lipschitz constant}

Let $(X,d)$ be a metric space. For a given continuous function $f:X\to\R$,
the \em pointwise Lipschitz constant \em at a non-isolated point $x_0\in X$ is defined as
$$
\mathrm{Lip} f(x_0):=\limsup_{\substack{x\to x_0\\ x\neq x_0}}
\frac{|f(x_0)-f(x)|}{d(x,x_0)}.
$$
If $x_0$ is an isolated point of $X$, we define $\mathrm{Lip} f(x_0)=0$.

This constant has been extensively used in the field of analysis on metric spaces as a substitute for the modulus of the derivative of a Lipschitz function defined in non-smooth settings. Note that, if $f$ is a continuous real function defined on $\R^n$ (or, more generally, on a Banach space) which is differentiable  at a point $x_0$, then $ \mathrm{Lip} f(x_0)=\|\nabla f(x_0)\|$. However, the pointwise Lipschitz constant  is always non-negative, so it is not so useful if one wants to determine ``descent'' or ``ascent'' directions (and detect minima or maxima). To this end, an {\em asymmetric} object is needed.

\begin{definition}[Metric slope]\em
For a function $f:X\to\R$, the \em metric slope \em at a point $x_0\in X$ is  $$
|\partial f|^-(x_0):=\limsup_{\substack{x\to x_0\\ x\neq x_0}}
\frac{\max\{f(x_0)-f(x),0\}}{d(x_0,x)}=\left\{\begin {array}{l} 0\quad \text{if $x_0$ is a local minimizer of $f$,}\\[10pt]
\displaystyle{\limsup_{\substack{x\to x_0\\ x\neq x_0}}\,\frac{f(x_0)-f(x)}{d(x_0,x)}}\quad \text{otherwise}.
\end{array}
\right.
$$\end{definition}

By definition, if $x_0$ is a local minimizer of $f$, then $|\partial f|^-(x_0)=0$. For example, the function $f(x)=|x|$ on $\R$ is not differentiable at $x_0=0$, has a minimum at $x_0=0$ and $|\partial f|^-(x_0)=0$. However notice that, in general, having $|\partial f|^-(x_0)=0$ at a point $x_0$ does not necessarily imply that $x_0$ is a minimum of $f$ (consider, for example, the function $f(x)=-x^2$ on $\R$).

The metric slope, also called {\em local slope}, {\em descendent slope} or {\em calmness rate}, was introduced in \cite{GMT} (see also \cite{AGS}) in connection with steepest descent evolutionary problems. The notation should not be confused with the {\em relaxed slope} $|\partial^- f|$ defined in \cite[Section 2.3]{AGS}.

An \em ascendent \em metric slope, denoted by $|\partial f|^+ (x_0)$, can be defined in a similar manner, replacing ${f(x_0)-f(x)}$ by $f(x)-f(x_0)$ in the numerator. In this case, we have that
$$
|\partial f|^+ (x_0)=|\partial (-f)|^- (x_0).
$$

The relation between metric slopes and the pointwise Lipschitz constant is clear. If $(X,d)$ is a metric space and $f:X\to \R$ is a function, it is straightforward to see that
\begin{equation}\label{Liprepresentation}
\mathrm{Lip} f(x_0)=\max\{|\partial f|^- (x_0),|\partial f|^+ (x_0)\}.
\end{equation}
Indeed, in order to obtain this, we just need to use that
$$
|f(x)-f(x_0)|=\max\{\{f(x)-f(x_0),0\},\{f(x_0)-f(x),0\}\}.
$$

The following simple lemma will be useful later on.

\begin{lemma}\label{differentiable}
If $f:[a,b]\to \R$ is differentiable at $x_0\in(a,b)$, then
 \[
 |f'(x_0)|=\Lip f(x_0)=|\partial f|^- (x_0)=|\partial f|^+ (x_0).
 \]
\end{lemma}

\begin{proof}
Fix $x_0\in(a,b)$. Let $A=\{x:f(x_0)\leq f(x)\}$ and $B=\{x:f(x_0)\geq f(x)\}$. Recall that
$$
|\partial f|^- (x_0)=\lim_{r\to 0}\sup_{d(x,x_0)<r}\frac{\max\{f(x_0)-f(x),0\}}{|x_0-x|}.$$
Observe that
\begin{equation}\label{AandB}
\sup_{\stackrel{d(x,x_0)<r}{x\in A}}\frac{\max\{f(x_0)-f(x),0\}}{|x_0-x|}=0\quad\text{and}\quad\sup_{\stackrel{d(x,x_0)<r}{x\in B}}\frac{\max\{f(x_0)-f(x),0\}}{|x_0-x|}=\sup_{\stackrel{d(x,x_0)<r}{x\in B}}\frac{|f(x_0)-f(x)|}{|x_0-x|}.\
\end{equation}
We now distinguish three cases:
\begin{enumerate}
\item If $x_0$ is an accumulation point of $A\cap B$, then by \eqref{AandB}, $|\partial f|^- (x_0)= |f'(x_0)|$.
\item If $x_0$ is not an accumulation point of $B$, there exists $r_{x_0}>0$ such that $B(x_0,r_{x_0})\cap B=\emptyset$, then $x_0$ is a local minimum and so $|\partial f|^- (x_0)= |f'(x_0)|=0$.
\item If $x_0$ is not an accumulation point of $A$, there exists $r_{x_0}>0$ such that $B(x_0,r_{x_0})\cap A=\emptyset$ and so $|\partial f|^- (x_0)= |f'(x_0)|$ (in this case $x_0$ is a local maximum so in particular $|\partial f|^- (x_0)= |f'(x_0)|=0$).
\end{enumerate}
A similar argument shows that $ |f'(x_0)|=|\partial f|^+ (x_0)$ and the conclusion follows.
\end{proof}

%See also \url{https://www.math.umd.edu/~yanir/OT/AmbrosioGigliDec2011.pdf}

%{\color{red} En el paper original de  \cite{GMT}  la local slope se denota por $|\nabla f|$. Me gustaría añadir lo que comentaba Jesús de que el $\mathrm{Lip}$ %está más asociado a las $\mathcal{C}^1$ y el $\mathrm{Lip}_a$ a las diferenciables.}

%If $(X,d)$ is a length space and $f$ is lower semicontinuos and geodesically convex (that is, convex along geodesics), then the metric slope admits the following representation
%\[
%|\partial f|^- (x_0)=\sup_{x\neq x_0}\frac{\max\{f(x_0)-f(x),0\}}{d(x_0,x)}.
%\]
%Moreover, $|\partial f|^- (x_0)$ is an upper gradient for $f$ (see \cite[Corollary 2.4.10]{AGS}). Recall that a non-negative Borel function $g$ on $X$ is an \em upper gradient \em of an extended real-valued function $f$ on X, if
%$$
%|f(\gamma(a))-f(\gamma(b))|\leq\int_{\gamma}g
%$$
%for every rectifiable curve $\gamma:[a,b]\rightarrow X$.
% For a Lipschitz function $f:X\to\R$ defined on a metric space, with Lipschitz constant $\LIP(f)$, we have
%\begin{equation}\label{desigualdadesslope}
%    |\partial f|^- (x) (\text{ or } |\partial f|^+(x))\leq \mathrm{Lip} f(x)\leq \mathrm{LIP}(f),
%\end{equation}

Of course, in general, the ascendent and descendent metric slopes can be different. For example,  the function $f:\R \to \R$ defined by $f(x)=0$ for $x\leq 0$, and $f(x)=-\sqrt{x}$ for $x>0$, satisfies that $0=|\partial f|^+ (0)<|\partial f|^- (0)=\Lip f(0)=+\infty$. On the other hand, the function $f:\R \to \R$ defined by $g(x)=-2x$ for $x\leq 0$, and $g(x)=-x$ for $x>0$, is a function for which $1=|\partial g|^- (0)<|\partial g|^+ (0)=\Lip g(0)=2$.

%{\color{red} Esti: Notice that for $f:X\to Y$ it is possible to define the pointwise Lipschitz constant but not the metric slope (unless $Y=\R$).}

\subsection{Lipschitz-type functions in quasi-metric spaces}

The notions of metric slope and pointwise Lipschitz constant can be easily brought to the general context of quasi-metric spaces. Both objects will appear as particular cases of this general setting. Let us recall the definition of quasi-metric space, where we lose the symmetric property of the ``distance''.

\begin{definition}[Quasi-metric space]\em
A \em quasi-metric space \em $(X,d)$ is a set $X\neq\emptyset$ and a function $d:X\times
X\rightarrow [0,\infty)$, called {\em quasi-metric}, such that:
\begin{enumerate}
\item $d(x,x)=0$ for all $x\in X$.
\item $d(x,y)=d(y,x)=0$ implies $x=y$.
\item $d(x,z)\leq d(x,y)+d(y,z)$ for all $x,y,z\in X$.
\end{enumerate}
\end{definition}

Non-reversible Finsler manifolds (that is, when the Finsler metric is in general non-reversible) are remarkable examples of quasi-metric spaces (see \cite{BCS} for extensive information about Finsler manifolds). Quasi-metric spaces are also called in literature {\em irreversible metric spaces} (see e.g. \cite{KZ}). Since quasi-metrics do not necessarily possess the symmetric property of a distance, they are also called {\em asymmetric spaces}. Given a quasi-metric $d$, one can always consider a symmetrized version $d^s$ (which is a true metric). The most common {\em symmetrizations} of a given quasi-metric $d$ are
\[
d^s:=\max\{d(x,y),d(y,x)\} \quad \text{ or }\quad d_{av}^s:=\frac{d(x,y)+d(y,x)}{2}.
\]

Given a quasi-metric space $(X, d)$, the {\em forward open ball} of center $x_0\in X$ and radius $r>0$ is defined as
$$
B^+(x_0, r):= \{x \in X \, : \, d(x_0, x) < r \}.
$$
We will always consider on $(X, d)$ the so-called {\em forward topology,}  which has as a base the family of forward open balls (the {\em backward topology} can be defined in a similar manner, as one might expect, through {\em backward open balls}). Recall that a sequence $(x_n)$ in $(X,d)$ {\em converges} to $x_0$ {\em in the forward topology} if,  and only if, $d(x_0,x_n)$ converges to $0$ in $(\R,|\cdot|)$. 
\

In the case of the real line $\R$ we have a natural asymmetric structure, given by the quasi-metric $d_u$ defined by
\begin{equation}\label{Rdu}
d_u(x,y)=\max\{y-x,0\} \text{ where } x,y\in\R.
\end{equation}
In fact, this quasi-metric is associated to an \em asymmetric norm \em on $\R$. As we will observe in Definition \ref{defasymmnorm}, an asymmetric norm $u$ on a real vector space differs from a norm in that it is positively homogeneous, but it does not necessarily satisfy $u(x)=u(-x)$. Here the quasi-metric $d_u$ above is associated to the asymmetric norm $u$ on $\R$  defined by
\[
u(x)=\max\{x,0\}\quad \text{ for every } x\in\R.
\]

For more information regarding quasi-metric spaces and asymmetric normed spaces we refer to \cite{cobzas} (see also \cite[Section 2]{DSV}).

In what follows, we define some Lipschitz-type functions between quasi-metric spaces. Recall that a function $f:(X, d_X) \to (Y, d_Y)$ between quasi-metric spaces is said to be {\em semi-Lipschitz} if there exists a constant $L\geq 0$ such that, for every $x, y\in X$
\begin{equation}\label{S-LIP}
d_Y (f(x), f(y)) \leq L \, d_X(x, y).
\end{equation}
The least constant $L$ satisfying the above inequality is called the {\em semi-Lipschitz constant} of $f$, and is denoted by $\mathrm{SLIP}f$. Note that, if $(Y, \rho) = (\R, d_u)$,  inequality \eqref{S-LIP} is equivalent to:
$$
f(x) - f(y) \leq L \, d(x,y).
$$
Semi-Lipschitz functions are the natural generalization of Lipschitz functions to the context of quasi-metric spaces. These  functions were first considered in \cite{RS} and have been widely studied since then (see e.g. \cite{cobzas} and references therein). Now we introduce the pointwise counterpart  of semi-Lipschitz functions.

\begin{definition}[Pointwise semi-Lipschitz functions]\label{poitwisesemi}\em
%\textcolor{red}{This definition should be stated in its $\varepsilon-\delta$ version, to avoid trouble with asking for the distances to be greater than $0$.}
Let $(X,d_X)$ and $(Y,d_Y)$ be quasi-metric spaces and $f:X\to Y$ be a function. We say that $f$ is {\em pointwise semi-Lipschitz} at  $x_0\in X$ if there exist $\alpha\geq 0$ and $\delta>0$ such that
\begin{equation}\label{pointslip}
    d_X(x_0,x)<\delta \implies  d_Y(f(x_0),f(x))\leq \alpha d(x_0,x).
\end{equation}

Note that this condition is trivially satisfied if $x_0$ is an isolated point. The infimum of all constants $\alpha$ satisfying inequality~\eqref{pointslip} is called the {\em pointwise semi-Lipschitz constant} of $f$ at $x_0$, and will be denoted by $\mathrm{SLip}f(x_0)$. If $x_0$ is an isolated point, we have that $\mathrm{SLip} f (x_0)=0$. If $x_0$ is non-isolated and $d(x_0,x)>0$ for all $x\in X\setminus \{x_0\}$,  $\mathrm{SLip}f(x_0)$ can be computed as
$$
\mathrm{SLip}f(x_0)= \limsup_{x \to x_0} \frac{ d_Y(f(x_0), f(x))}{d_X (x_0, x)}.
$$
A function $f:X\to Y$ is said to be \em pointwise semi-Lipschitz \em if it is pointwise semi-Lipschitz at every point, that is, $\mathrm{SLip} f (x)<+\infty$ for every $x\in X$.
\end{definition}

\begin{remark}\em
Of course, every semi-Lipschitz function is pointwise semi-Lipschitz. Furthermore, note that if  $f:(X, d_X) \to (Y, d_Y)$ is pointwise semi-Lipschitz at a point $x_0$, then $f$  is forward-forward continuous at  $x_0$ (i.e., with respect to the forward topologies of $X$ and $Y$). 
\end{remark}

\begin{remark}\label{generalizacionpointwise}\em
    It is easy to check that this definition encompasses both, the definitions of the pointwise Lipschitz constant and the metric slope in the following way:
\begin{itemize}
    \item If $(X,d)$ is a metric space and $(Y, d_Y)=(\R,|\cdot|)$, then  for every $x\in X$ we have that
    $$
    \mathrm{SLip} f(x)=\mathrm{Lip} f(x).
    $$

    \item If $(X,d)$ is a metric space and $(Y,d_Y)=(\R,d_u)$, then for every $x\in X$ we have that
    $$
    \mathrm{SLip} f(x)=|\partial f|^+ (x)=|\partial (-f)|^- (x).
    $$
    Therefore, for $f:X\to(\R,d_u)$, the formula~\eqref{Liprepresentation} can be rewritten as
\[
\mathrm{Lip} f(x)=\max\{\mathrm{SLip} f(x),\mathrm{SLip} (-f)(x)\}.
\]
\end{itemize}
\end{remark}

\noindent {\bf Convention:} In the case of real-valued functions $f: X \to \R$ defined on a quasi-metric space $X$, we set by convention that, when computing $\mathrm{Lip} f(x)$, $\R$ is assumed to carry the usual metric, and when computing $\mathrm{SLip} f(x)$, $\R$ is endowed with the quasi-metric $d_u$.

\begin{example}\label{funsemiLip}\em
 For a function $f:(X, d) \to (\R,|\cdot|)$ on a quasi-metric space $X$, a straightforward computation shows that $f$ is upper semicontinuous  at any point $x_0\in X$ where $\mathrm{SLip}~f(x_0)<+\infty$. On the other hand,  pointwise semi-Lipschitz functions do not need to be continuous for the usual metric on $\R$. Consider for example the function $f:(\R,|\cdot|)\to \R$ defined by
\begin{equation*}\label{ex1}
f(x)=\left\{\begin {array}{ll} 1&\text{if }x\geq 0\\[10pt]
0&\text{if } x< 0.
\end{array}
\right.
\end{equation*}
It is clear that  $f$ is discontinuous for the usual metric,  but
\[
\mathrm{SLip} f(x_0)=\limsup_{\substack{x\to 0\\ x\neq 0}}
\frac{\max\{f(x)-f(x_0),0\}}{|x-x_0|}=0\,\,\text{ at every }x_0\in\R.
 \]
\end{example}

\begin{remark}\em
If $(X,d)$ is a  {\bf metric space}, the space of Lipschitz functions ${f:(X, d) \to (\R,|\cdot|)}$ coincides with the space of semi-Lipschitz functions $f:(X, d)\to (\R,d_u)$. Indeed, since for any $t,s\in \mathbb{R}$ we have ${\max\{t-s,0\}\leq|t-s|}$, any Lipschitz function will also be semi-Lipschitz with respect to $d_u$, and any $f$ satisfying $\max\{f(y)-f(x),0\}\leq L d(x,y)$ must also satisfy $\max\{f(x)-f(y),0\}\leq L d(x,y)$ due to the symmetry of $d$. However, example~\ref{funsemiLip} shows that this is not the case for pointwise semi-Lipschitz functions (since $\mathrm{Lip} f(0)=+\infty$).
\end{remark}

\begin{remark}\label{remarkconstantes}\em
For a function $f:(X,d_X)\to (Y,d_Y)$ between quasi-metric spaces, the following inequalities hold,  where $\mathrm{SLip}_{d_1,d_2}$ is computed using the quasi-metric $d_1$ on $X$ and the quasi-metric $d_2$ on $Y$. For any $x\in X$,
\begin{itemize}
    \item $\mathrm{SLip}_{d_X^s, d_Y} f(x)\leq\mathrm{SLip}_{d_X, d_Y} f(x)\leq \mathrm{SLip}_{d_X, d_Y^s} f(x),$

    \item $\mathrm{SLip}_{d_1,d_2} f(x)\leq \mathrm{SLIP}_{d_1,d_2}f,$ for $d_1=\{d_X, d_X^s\}$ and $d_2=\{d_Y, d_Y^s\}$.
\end{itemize}
\end{remark}

In the recent years, there have been several developments regarding the use of metric slopes in \emph{determination theorems}, that is, results guaranteeing that functions sharing the same metric slopes are equal up to a constant (see \cite{DS} for an example of such a result). The recent work \cite{DMS} extended the framework of these determination theorems beyond metric spaces using the notion of \emph{abstract descent moduli}, which generalize the usual definition of metric slope of a function, and which does not rely on an underlying metric on the domain of the function. It is easy to check that, for any quasi-metric space $(X,d)$, the operator 
$$ \begin{array}{cccl}
   T:&\mathbb{R}^X&\longrightarrow &[0,\infty]^X  \\
     &f &\longrightarrow  &\SLip (-f)
\end{array} 
$$
\noindent satisfies the three conditions of Definition~3.1 of \cite{DMS}. In other words, the semi-Lipschitz constant of minus a function is an abstract descent modulus, and therefore, the results of Section~3.2 of \cite{DMS} can be applied to $T[f]=\SLip (-f)$. In particular, Theorem~3.5 states that continuous and coercive functions on any quasi-metric space are determined by their pointwise semi-Lipschitz constants and critical values (see \cite{DMS} for the precise definitions). 

Another determination result for descent moduli was shown in \cite{DLS}, where the authors were able to remove the need for coercivity of the involved functions, at the cost of requiring a certain relation between the descent modulus $T$ and an underlying metric $D$ on the set $X$. We shall see that this relation, called \emph{metric compatibility}, is satisfied by the operator $T[f]=\SLip (-f)$ under a partial symmetry assumption on the quasi-metric space $(X,d)$. We recall the definition of \emph{strong metric compatibility} from \cite{DLS}.

\begin{definition}[Metric compatibility]\label{metriccompatibilty}\em
An abstract descent modulus (see \cite{DMS}) $T$ is said to be metrically compatible with a metric $D$ if for every $\rho>0$ there exists a strictly increasing continuous function $\theta_\rho:[0,\infty)\to [0,\infty)$, with $\theta_\rho(0)=0$ and $\lim_{t\to\infty}\theta_\rho(t)=+\infty$ such that for every $f,g\in \mathrm{dom}(T)$, $x\in \mathrm{dom}(g)$ and $\delta >0$, it holds:

$$ T[f](x)<\delta < T[g](x)\implies \exists z\in \mathrm{dom}(g):\left\{\begin{array}{c}
    \max\{f(x)-f(z),0\}<(1+\rho)  \max\{g(x)-g(z),0\} \\
    \text{ and}\\
    \theta_\rho(\delta)D(x,z)<g(x)-g(z)
    
\end{array}\right. $$ 
\end{definition}

\begin{proposition}[$\SLip (-f)$ is metrically compatible]\label{SLipmetricallycompatible}
Let $(X,d)$ be a uniformly quasi-symmetric quasi-metric space (see forthcoming Definition~\ref{unifquasisym}). Then, the descent modulus ${T[f]=\SLip (-f)}$ is metrically compatible with the symmetrized metric $D=d^s$. 
\end{proposition}
\begin{proof}
Consider functions $f$ and $g$, $\rho>0$, $\delta>0$ and $x\in X$ such that 
$$\SLip (-f)(x) <\delta <\SLip (-g)(x)<+\infty.$$ 

As a consequence of $(X,d)$ being uniformly quasi-symmetric, we have that $d(x_0,x)>0$ for all $x\neq x_0$, so we can write $\SLip h(x)=\limsup_{x\to x_0}\frac{d_u(h(x_0),h(x))}{d(x_0,x)}$ for any function $h$.

By definition of the pointwise semi-Lipschitz constant, there exists $r>0$ small enough such that 

$$ \sup_{y\in B^+(x,r)}\frac{\max\{-f(x)+f(y),0\}}{d(x,y)}<\delta <\sup_{y\in B^+(x,r)}\frac{\max\{-g(x)+g(y),0\}}{d(x,y)}. $$

Moreover, we can choose $z\in B^+(x,r)$ such that 
$$\frac{\max\{f(x)-f(z),0\}}{d(x,z)}<\delta <\frac{\max\{g(x)-g(z),0\}}{d(x,z)},$$
\noindent which readily implies the first condition of Definition~\ref{metriccompatibilty}, and also that $g(x)-g(z)>0$.

Now, since $(X,d)$ is pointwise quasi-symmetric, we have that $D(x,z)\leq K d(x,z)$ whenever $d(x,z)$ is small enough, where ${K:= \sup_{x \in X} \SLip id_{d,D}(x) < +\infty}$ (see Definition~\ref{unifquasisym}). Thus, 
$$\frac{\delta}{K} <\frac{g(x)-g(z)}{D(x,z)},$$
\noindent and therefore the second condition of Definition~\ref{metriccompatibilty} is satisfied for the function $\theta_\rho(\delta)=\tfrac{\delta}{K}$. 
\end{proof}

Proposition~\ref{SLipmetricallycompatible} along with \cite[Theorem 3.6]{DLS} yield that for any uniformly quasi-symmetric quasi-metric space $(X,d)$ which is also \emph{bicomplete} (that is, a quasi-metric space such that $(X,d^s)$ is a complete metric space), continuous and bounded from below functions are determined by the values of ${T[f]=\SLip (-f)}$, along with the corresponding \textit{critical points and asymptotically critical sequences} (see \cite{DLS} for the precise definitions). 

\
For {\bf metric} spaces $(X,d_X)$ and $(Y,d_Y)$, we recall the definition of the space $D(X, Y)$ of pointwise Lipschitz functions, which was studied in \cite{DJ}:  \begin{itemize}
\item[$\diamond$] $D(X,Y) =\{f:X\to Y\, :\, \sup_{x\in X}\mathrm{Lip} f(x)<+\infty\}.$
\end{itemize}
Let us point out that pointwise Lipschitz functions are in fact continuous \cite[Lemma 2.2]{DJ}. In the case that $(Y, d_y)=(\R, \vert \cdot \vert)$ it will be omitted in the notation, and the corresponding space is denoted by $D(X)$.

Now we are going to introduce two analogous spaces of pointwise semi-Lipschitz functions between quasi-metric spaces.

\begin{definition}[Pointwise semi-Lipschitz spaces]\label{DSL}\em
Let $(X,d_X)$ and $(Y,d_Y)$ be quasi-metric spaces. We define the following function spaces:
\begin{itemize}
\item[$\diamond$] $C_{SL}(X, Y)$ is defined as the space of all pointwise semi-Lipschitz functions from $(X,d_X)$ to $(Y,d_Y)$ which are continuous for the forward topology on $X$ and the symmetrized topology in $Y$ (that is, $(X,d_X)$-to-$(Y,d_Y^s)$ continuous.)
\item[$\diamond$] $D_{SL}(X, Y)$ is defined as the space of all functions $f\in C_{SL}(X, Y)$ such that
$$\|\SLip f\|_{\infty}=\sup_{x \in X} \SLip \, f(x)< +\infty.$$
\item[$\diamond$] In the case that $(Y, d_Y)=(\R, d_u)$, the target space will be omitted in the notation, and the corresponding spaces of real-valued functions will be denoted, respectively, by $C_{SL}(X)$ and $D_{SL}(X)$.
\end{itemize}
\end{definition}

We now show the stability of these function spaces by composition.

\begin{lemma}\label{lemmacompos} Let $(X, d_X), (Y, d_Y),$ and $(Z, d_Z)$ be quasi-metric spaces. Then:
\begin{itemize}
\item[(a)] If $f \in C_{SL}(X, Y)$ and $g \in C_{SL}(Y, Z)$, then $g \circ f \in C_{SL}(X, Z)$ and, for every $x\in X$,
$$
\SLip \, g \circ f(x) \leq \SLip \, g (f(x)) \cdot \SLip \, f(x).
$$
\item[(b)] If $f \in D_{SL}(X, Y)$ and $g \in D_{SL}(Y, Z)$, then $g \circ f \in D_{SL}(X, Z)$.
\end{itemize}
\end{lemma}
\begin{proof}
It is clear that $(b)$ follows from $(a)$, so let us prove $(a)$. We first note that $g \circ f$ is continuous for the forward topology in $X$ and the symmetrized topology in $Z$. Now, consider $x_0 \in X$ and let $y_0 = f(x_0)$. We know that there exist $\alpha > 0$ and $\delta > 0$ such that if $d_Y(y_0, y) < \delta$, then
$$
d_Z(g(y_0), g(y)) \leq \alpha d_Y(y_0, y).
$$
We also know that there exist $\beta > 0$ and $\delta' > 0$ such that if $d_X(x_0, x) < \delta'$, then
$$
d_Y(f(x_0), h
f(x)) \leq \beta d_X(x_0, x).
$$
Choosing $0 < \delta'' < \min\{\delta', \frac{\delta}{\beta}\}$, we obtain that if $d_X(x_0, x) < \delta''$, then $d_Y(h(x_0), h(x)) \leq \beta \cdot \delta' < \delta$. Therefore,
$$
d_Z(g(f(x_0)), g(f(x))) \leq \alpha \beta d_X(x_0, x).
$$
Then $g \circ f$ is pointwise semi-Lipschitz at $x_0$ and furthermore
$$
\SLip \, g \circ f(x_0) \leq \SLip \, g (y_0) \cdot \SLip \, f(x_0).
$$
\end{proof}

%%%%%%%%%%%%%%%%%%%%%
%%%%%%%%%%%%%%%%%%%%%%%%%%%%%%%%%%%%%%%%%%
\begin{definition}[Pointwise semi-Lipschitz homeomorphism]\em
We say a bijection between quasi-metric spaces ${h:(X, d_X) \to (Y, d_Y)}$ is a $C_{SL}$-{\em homeomorphism} whenever ${h\in C_{SL}(X, Y)}$ and \\${h^{-1}\in C_{SL}(Y, X)}$. In the same way,  $h$ is said to be a  $D_{SL}$-{\em homeomorphism} if $h\in D_{SL}(X, Y)$ and $h^{-1}\in D_{SL}(Y, X)$.
\end{definition}
%%%%%%%%%%%%%%%%%%%%%%%%%%%%%%%%%%%%%%%%%%%%%%%%%%%%%%%%%%%%%%%%%%%%%%%%%%%%%%%%%%%%%%%%%%%%%%%%%%%%%%%%%%%%%%%%%%%%%%%%%%%

%%%%%%%%%%%%%%%%%%

For a {\bf metric} space $(X, d)$, since every real-valued function on $X$ satisfies $\mathrm{SLip} f(x) \leq \mathrm{Lip} f(x)$, it is clear that $D(X) \subset D_{SL}(X)$. In our next example we  show a  metric space $X$ for which $D_{SL}(X) \neq D(X)$.

\begin{example}\label{funcontDSL}\em
Consider on the interval $X=[0,1]$ the snowflake distance $d(x, y)= \vert x-y \vert^{1/2}$. Select a point $a\in X$, and choose a sequence of different points $(a_n)$ in $X$ converging to $a$, and a sequence of small enough radii $(r_n)$, such that: $0<r_n<\frac{1}{n} d(a_n, a)$, the open balls $B_d(a_n, r_n)$ are pairwise disjoint, and each $x\neq a$ has a neighborhood $V^x$ which meets only a finite number of balls $B_d(a_n, r_n)$.

Note that any Lipschitz function $f:([0,1], \vert \cdot \vert) \to \mathbb R$ satisfies that $\mathrm{Lip} f(x)=0$ for every $x\in X$. Indeed, if
$$
\vert f(x) - f(y)\vert \leq K \vert x-y\vert
$$
Then
$$
\limsup_{y \to x}\frac{\vert f(x) - f(y)\vert}{d(x, y)} \leq \limsup_{y \to x}\frac{K \vert x-y\vert}{\vert x-y\vert^{1/2}} =0.
$$

Now, if for each $n$ we fix some $k_n>0$, we can select a continuous function $f_n:X\to \mathbb R$, which is Lipschitz for the euclidean distance in $X$, such that $0 \geq f_n\geq - k_n d(a_n, a)$ on $X$, $f_n(a_n)= - k_n d(a_n, a)$, and the support of $f_n$ is contained into $B(a_n, r_n)$. Now consider
$$
f=\sum_{n=1}^{\infty} f_n.
$$
The function $g$ is well-defined and continuous on $X\setminus \{a\}$, since the sum is locally finite. If we assume in addition that
$$
\lim_{n\to \infty} k_n d(a_n, a) =0,
$$
we obtain that $f$ is continuous on $X$. Furthermore, we have that $\mathrm{Lip} f(x)=0$ for each $x\neq a$. Thus $\mathrm{SLip} f(x)=\vert \partial f \vert ^{+}(x)=0$ for $x\neq a$. Since $a$ is a local (and global) maximum of $f$, we see that also $\mathrm{SLip} f(a)=\vert \partial f \vert ^{+}(a)=0$.

On the other hand,
$$
\mathrm{Lip} f(a) = \limsup_{n\to \infty} \frac{\vert f(a) - f(a_n)\vert}{d(a_n, a)} = \sup_n k_n.
$$
If we choose a constant sequence $k_n =k$ we obtain an example where
$$
\sup_{x\in X} \mathrm{SLip} f(x) = 0 < k = \sup_{x\in X} \mathrm{Lip} f(x) < +\infty.
$$
On the other hand, if we choose a sequence $(k_n)$ tending to infinity, for example $k_n = d(a_n, a)^{-1/2}$, we obtain an example where $\mathrm{SLip} f(x) = 0$ and  $\mathrm{Lip} f(x)< +\infty$ for every $x\in X$, but
$$
\sup_{x\in X} \mathrm{Lip} f(x) = +\infty.
$$
\end{example}

\begin{remark} \label{flat} \em
Note that an analogous construction  can be carried out on a compact metric space $(X, d)$ provided:

(i) $(X, d)$ contains an accumulation point $a$.

(ii) $(X,d)$ is purely 1-unrectifiable, that is, it  contains no bi-Lipschitz copy of any compact
subset of $\R$ with positive measure.

%We refer to \cite{AGPP} for the relevant definitions and results. 
Recall that a function $f:X \to \R$ on a metric space $X$ is said to be {\em locally flat} if, for every $p \in X$ and every $\varepsilon>0$ there exists $\delta>0$ such that, if $x, y \in B_d(p, \delta)$ then $\vert f(x)-f(y)\vert \leq \varepsilon d(x, y)$. As a consequence of \cite[Theorem A]{AGPP}, every compact and purely 1-unrectifiable metric space $X$ satisfies that the space $\lip(X)$ of  locally flat Lipschitz functions $f:X\to \R$  {\em separates points} of $X$.  Note that $\lip(X)$ is a unital subalgebra of the space $C(X)$ of continuous real functions on $X$, so by the classical Stone-Weierstrass theorem we have that $\lip(X)$ is uniformly dense in $C(X)$. In particular $\lip(X)$ separates disjoint closed subsets of $X$. It is easily seen that this allows a construction similar to the one in Example \ref{funcontDSL}.
\end{remark}

Our next objective is to provide geometric conditions on a metric space $(X, d)$ under which $D(X)=D_{SL}(X)$.

Recall that the {\em length} of a continuous curve $\gamma :[a, b]\to X$ is defined as
$$
\ell (\gamma) = \sup \left\{ \sum_{i=1}^m d(\gamma (t_i), \gamma (t_{i-1}) \right\},
$$
where the supremum is taken among all partitions $\{ a=t_0<t_1< \dots t_m=b \}$ of the interval $[a, b]$. The curve $\gamma$ is {\em rectifiable} if
$\ell(\gamma) < + \infty$.

The space $(X,d)$ is said to be  {\it quasi-convex} if there exists a constant $K>0$ such that, for every $x, y \in X$, there is a curve $\gamma$ in $X$ from $x$ to $y$ whose length satisfies $\ell(\gamma) \leq K\, d(x,y)$.

\begin{definition}[uniformly locally radially quasi-convex]\label{defuniflocallytadiallyqc}\em
A metric space $(X,d)$ is called \em uniformly locally radially quasi-convex \em if there exists a constant $K>0$ such that for every $x_0\in X$ there exists a neighborhood $U_{x_0}$ of $x_0$ such that for every $y\in U_{x_0}$ there exists a rectifiable curve $\gamma$ in $U_x$ connecting $x$ and $y$ and such that $\ell(\gamma)\leq Kd(x,y)$.
\end{definition}

We first need the following technical lemma, which is analogous to the one presented in \cite[Lemma 2.3]{DJ} for functions with bounded pointwise Lipschitz constant.

\begin{lemma}{\bf(Length lemma)}\label{lemmacurva}
Let $(X,d)$ be a metric space and let $f\in D_{SL}(X)$.  Let $x,y\in X$ and suppose there exists a rectifiable curve $\gamma:[a,b]\to X$ such that $\gamma(a)=x$ and $\gamma(b)=y$. Then,
$$ f(y)-f(x)\leq \|\mathrm{SLip} f\|_\infty \ell(\gamma).$$
\end{lemma}
\begin{proof}
Let $K=\|\mathrm{SLip} f\|_\infty<+\infty$. For $\varepsilon>0$, let us denote $K'=K+\varepsilon$. Since ${\mathrm{SLip}~f(\gamma(a))<K'}$, there exists $\delta>0$ such that, whenever $d(\gamma(a),x)<\delta$, we have that
$$f(\gamma(a))-f(x)\leq K'd(\gamma(a),x). $$
By continuity of $\gamma$, there exists $t^*\in (a,b]$ such that $d(\gamma(a),\gamma(t^*))<\delta$, and therefore
\begin{equation*}
  f(\gamma(a))-f(\gamma(t^*))\leq K'd(\gamma(a),\gamma(t^*))\leq K'\ell(\gamma|_{[a,t^*]}).
\end{equation*}
Let us consider the set
$$A=\{t\in (a,b]~:~f(\gamma(a))-f(\gamma(t))\leq K'\ell(\gamma|_{[a,t]})\},$$
\noindent which is clearly non empty (as $t^*\in A$) and bounded above, so we can consider $s=\sup(A)$. Let us check that $s$ belongs to $A$. By definition of $s$, there exists a sequence $(t_n)\subset A$ such that $(t_n)\to s$ and $f(\gamma(a))-f(\gamma(t_n))\leq K'\ell(\gamma|_{[a,t_n]})$. By continuity of $f$, we conclude that $f(\gamma(a))-f(\gamma(s))\leq K'\ell(\gamma|_{[a,s]})$.
Next, we shall prove that $s=b$. If this were not the case, we would have $a<s<b$, and since $\mathrm{SLip}~f(\gamma(s))<K'$, we can take $t^*\in(s,b]$ satisfying
$$f(\gamma(s))-f(\gamma(t^*))\leq K'\ell(\gamma|_{[s,t^*]}).$$
Then,
\begin{align*}
    f(\gamma(a))-f(\gamma(t^*))&=f(\gamma(a))-f(\gamma(s))+f(\gamma(s))-f(\gamma(t^*))\\
    &\leq K'\ell(\gamma|_{[a,s]})+K'\ell(\gamma|_{[s,t^*]})\\
    &= K'\ell(\gamma|_{[a,t^*]}),
\end{align*}
\noindent which implies $t^*\in A$, contradicting the fact that $s=\sup (A)$. Having proved that $s=b$, the fact that $s\in A$ yields the desired result.
\end{proof}

A direct consequence of Lemma \ref{lemmacurva} and  \cite[Lemma 2.3]{DJ} is the following result. Here $\LIP(X)$ denotes, as usual, the space of all real-valued Lipschitz functions on a metric space $X$.

\begin{corollary}\label{quasiconvex}
Let $(X,d)$ be a metric space.
\begin{enumerate}
\item If $(X,d)$ is a quasiconvex space, then $\LIP(X)=D(X)=D_{SL}(X)$.
\item If $(X,d)$ is a uniformly locally radially quasiconvex space, then $D(X)=D_{SL}(X)$.
\end{enumerate}
\end{corollary}

The last lemma of this section will not be used in the rest of the article but might be of independent interest. Recall that a non-negative Borel function $g$ on $X$ is an \em upper gradient \em of $f:X\to\R$, if
$$
|f(\gamma(a))-f(\gamma(b))|\leq\int_{\gamma}g,
$$
for every rectifiable curve $\gamma:[a,b]\rightarrow X$. For a Borel measurable function $g:\gamma([a,b])\rightarrow[0,\infty]$, we define $
\int_{\gamma}g:=\int_0^{\ell(\gamma)}g(\widetilde{\gamma}(t))dt$,
where $\widetilde{\gamma}:[0,\ell(\gamma)]\rightarrow X$ is the arc-length parametrization of $\gamma$.

For a general function $f:X\to\R$, the metric slope $|\partial f|^-$ is only a {\em weak} upper gradient for $f$ (see \cite[Definition 1.2.2]{AGS} and \cite[Theorem 1.2.5]{AGS}). The following lemma shows that, under additional geometrical assumptions on $X$, $|\partial f|^-$ is also an upper gradient.

\begin{lemma}\label{slopegradient}
Let $(X,d)$ be a uniformly locally radially quasiconvex space. If $f\in D_{SL}(X)$, then $|\partial f|^+$ and $|\partial f|^-$ are upper gradients of $f$.
\end{lemma}

\begin{proof}
Let $\gamma:[a,b]\rightarrow X$ be a rectifiable curve in X, parametrized by arc-length, and $f\in D_{SL}(X)$. In particular, $\gamma$ is $1$-Lipschitz (see \cite[Theorem 3.2]{Ha}). By Corollary \ref{quasiconvex}, $f\in D(X)$ and so $f\circ\gamma\in D([a,b])$.  By Stepanov's Differentiability Theorem, $f\circ\gamma$ is differentiable a.e. and, by Lemma \ref{differentiable}, $|(f\circ\gamma)'(t)|=|\partial (f\circ\gamma)|^+(t)=|\partial (f\circ\gamma)|^-(t)$ at every point $t\in [a,b]$  where $f\circ\gamma$ is
differentiable. Now, we deduce that
 $$
 |f(\gamma(a))-f(\gamma(b))|\leq\Big|\int_a^b(f\circ\gamma)^{'}(t)dt\Big|\leq\int_a^b|\partial (f\circ\gamma)|^+(t)\,dt=\int_a^b|\partial f|^+(\gamma(t))\,dt.
 $$
\end{proof}

\subsection{Pointwise symmetry of quasi-metric spaces}

%Buscamos la manera de extender el Teorema 5.36 (y siguientes) al caso de espacios asim\'etricos. Para ello, seguramente ser\'a necesario suponer alg\'un tipo de ``simetr\'ia parcial'' en los espacios. En concreto, yo creo que el Teorema 5.36 se puede extender al caso de que los espacios cuasi-m\'etricos $(X, d_X)$ e $(Y, d_Y)$ sean puntualmente casi-sim\'etricos en el sentido siguiente:

Dealing with quasi-metric spaces, it is often useful to have some kind of partial symmetry, and in particular we will need this in order to obtain our Banach-Stone type results.  Given a quasi-metric space $(X, d)$, the {\em reverse quasi-metric} is defined by $\bar d (x, y):= d(y, x)$. Of course, not every quasi-metric space exhibits the same behavior with respect to its reverse quasi-metric. As a trivial example, we have metric spaces, where $d$ and $\bar d$ coincide. The opposite phenomenom occurs on $(\mathbb{R},d_u)$ (see \eqref{Rdu}), where $d_u$ and $\bar{d_u}$ are never comparable, as ${\bar{d_u}}(x,y)=0$ whenever $d_u(x,y)>0$. Intermediate examples can be constructed by taking $\alpha \in (0,\infty)$, and defining following the quasi-metric on $\mathbb{R}$
$$
\rho_\alpha(x,y)=\begin{cases}
\begin{array}{ccc}
&\alpha (y-x)&\text{ if }y\geq x \\
&x-y&\text{ if }x\geq y
\end{array}
\end{cases}
$$
\noindent so that $\rho_\alpha$ and $\bar{\rho_\alpha}$ are somewhat ``equivalent''.

The following notion, which quantifies how asymmetric a quasi-metric space is, was introduced independently by Shen and Zhao in \cite{SZ} and by Bachir and Flores in \cite{BF}.

\begin{definition}[Index of symmetry]

\em Let $(X,d)$ be a quasi-metric space. The \em index of symmetry \em of $(X,d)$ is defined by
\[
c_d(X)=\inf_{d(x,y)>0}\frac{d(y,x)}{d(x,y)}\in[0,1].
\]
\end{definition}

Clearly, the class of quasi-metric spaces with index of symmetry $1$ is exactly the class of metric spaces. It is also easy to check that $c_{d_u}(\mathbb{R})=0$ and that $c_{\rho_\alpha}(\mathbb{R})=\min\{\alpha,\alpha^{-1}\}$. In other words, the smaller $c_d(X)$ is, the {\em less symmetric} the quasi-metric space is.

If $(\mathcal{X},d_\mathcal{X})$ is a connected Finsler manifold, then every point $p\in \mathcal{X}$ has a neighborhood $U_p$ such that the index of symmetry of the subspace $(U_p,d_\mathcal{X}|_{U_p})$ is strictly positive. We refer to \cite{BCS} for details. It follows that the index of symmetry of every compact and connected Finsler manifold is strictly positive. This does not necessarily hold in the non-compact case, as we can see in the following example. We refer to  \cite{DJV} for undefined terms and details needed in the next example.

\begin{example}[Finsler manifold with index of symmetry $0$]\label{finslerindex0} \em
Consider $\mathbb{R}$, endowed with the Finsler structure $F_{(x, v)} = \vert v \vert - d \phi(x)(v)$, where $\phi : \mathbb R \to \mathbb{R}$ is given by
$$
\phi (x) := \int_0^x \frac{t^2}{1+t^2} \, dt.
$$
Note that $(\mathbb{R}, F)$ is a Randers space, since $\vert \phi'(x) \vert < 1$ for every $x\in \mathbb{R}$. It is easy to see that the associated Finsler distance is $d_{F} (x, x')= \vert x-x'\vert + \phi (x) - \phi(x')$. We now check that $c_{d_F}(\mathbb{R})=0$. First consider the points $x$ and $x'=x+1$. Using the definitions of $\phi$ and $d_F$, we can compute
$$ c_{d_F}(\mathbb{R})\leq \frac{d_F(x,x+1)}{d_F(x+1,x)}=\frac{\arctan(x+1)-\arctan(x)}{2+\arctan(x)-\arctan(x+1)},$$
\noindent which converges to $0$ when $x\to +\infty$.
\end{example}

In the work \cite{OZ} (as well as in \cite{SZ}) the quantity $\lambda_d(X)=(c_d(X))^{-1}\in [1, +\infty]$ is used instead, under the name \em reversibility \em of $(X,d)$. Notice that, for each $x,y\in X$,
\[
c_d(X) d(x,y)\leq d(y,x)\leq\lambda_d(X) d (x,y).
\]
So, if $c_d(X)>0$, then $id: (X, d) \to (X, \bar{d})$ is bi-semi-Lipschitz.

\begin{definition}[Pointwise index of symmetry]\label{pointwisequasisymm}\em
Let $(X,d)$ be a quasi-metric space. We say that $(X,d)$ is {\it quasi-symmetric} at $x_0\in X$ if the identity map $id: (X, d) \to (X, \bar{d})$ is pointwise semi-Lipschitz at $x_0$. In this case, the {\it pointwise index of symmetry} of $(X,d)$ at $x_0$ is defined by
$$
\sigma(x_0):= \SLip \text{id}\,(x_0).
$$
Observe that if $x_0$ is an isolated point then $\sigma(x_0)=0$. Furthermore, if $x_0$ is non-isolated and $d(x_0,x)>0$ for all $x\in X\setminus \{x_0\}$, then
$$
\sigma (x_0) = \limsup_{x \to x_0} \frac{ d(x, x_0)}{d(x_0, x)}.
$$
We say that $(X,d)$ is {\it pointwise quasi-symmetric} if $\sigma(x_0)<+\infty$ for each $x_0\in X$.
\end{definition}

Of course, every Finsler manifold is pointwise quasi-symmetric since, in this case, each point has a neighborhood where the forward and backward distances are bi-semi-Lipschitz equivalent (see \cite[Lemma 6.2.1]{BCS}). However, there are other examples of pointwise quasi-symmetric spaces where this local bi-semi-Lipschitz equivalence does not hold. We describe one in the next example.

\begin{example}\em
For each $n\in \N$, let us consider the quasi-metric space $(S_n,d_n)$ where $S_n=[0,\infty)$ and
$$
d_n (p,q)=\left\{\begin {array}{ll} n(q-p)&\text{if}\,\,p\leq q\\[10pt]
p-q&\text{if}\,\,q\leq p.
\end{array}
\right.
$$
Notice that, for each $p\geq 0$, $d_n(p,0)=p$ whereas $d_n(0,p)=np$.

We define the quasi-metric space $(X,d)$ as follows. Let  $X$ be the disjoint union of the family $\{(S_n,d_n):n\in\N\}$, that is
$$
X=\bigsqcup_n S_n,
$$
where we identify the origin of every $S_n$ to a single point, which we still denote by $0$, and
$$
d(p,q)=\left\{\begin {array}{ll} d_n(p,q)&\text{if}\,\,p,q\in S_n\\[10pt]
d_n(p,0)+d_m(0,q)&\text{if} \,\, p\in S_n, q\in S_m , \, \,  n\neq m .
\end{array}
\right.
$$
Recall that the elements of the disjoint union are ordered pairs $(p,n)=p_n$, where $p\in S_n$. Here $n$ serves as an auxiliary index that indicates which $S_n$ the element $p$ comes from.
Notice that, for each $n\in\N$,
\[
\frac{ d(p_n, 0)}{d(0, p_n)}=\frac{p_n}{np_n}=\frac{1}{n}
\]
so
\[
\sigma(0)=\limsup_{x \to 0} \frac{ d(x, 0)}{d(0, x)}\leq 1.
\]
Also, if $0<p_n<q_n$, then $\frac{d(q_n,p_n )}{d(p_n,q_n)}=\frac{q_n-p_n}{n(q_n-p_n)}=\frac{1}{n}$. On the other hand, if $0<q_n<p_n$, then ${\frac{d(q_n,p_n )}{d(p_n,q_n)}=\frac{n(p_n-q_n)}{p_n-q_n}=n}$ and so $\sigma(p_n)<\infty$ for each $n\in\N$. Therefore $c_d(X)=0$ but $(X,d)$ is pointwise quasi-symmetric. In particular, $id: (X, d) \to (X, \bar{d})$ is pointwise semi-Lipschitz but $id: (X, \bar{d}) \to (X, d)$ is not even continuous (since forward and backward topology in $X$ are different).
Indeed, consider the sequence $\{x_k\}_k$ where $x_k=\frac{1}{k}$ in $S_k$. One has that $\bar{d}(0,x_k)=d(x_k,0)=\frac{1}{k}\to 0$ when $k\to\infty$ whereas $d(0,x_k)=kx_k=1$. Therefore, for each $r>0$, $\sup_{x\in B(0,r)}\sigma(x)=\infty$ , as opposed to the Finsler case.
\end{example}

\begin{remark}\em It is easy to see that the identity $id: (X, d) \to (X, \bar{d})$ is pointwise semi-Lipschitz at a point $x_0$ if and only if the identity $id: (X, d) \to (X, d^s)$ is pointwise semi-Lipschitz at $x_0$.
\end{remark}

Now we characterize the "pointwise quasi-symmetry" of the space:

\begin{proposition}\label{propequividentity}
Let $(X, d)$ be a quasi-metric space. The following statements are equivalent:
\begin{itemize}
\item[(a)] $(X, d)$ is pointwise quasi-symmetric.
\item[(b)] The identity map $id: (X, d) \to (X, d^s)$ is a $C_{SL}$-homeomorphism.
\end{itemize}
\end{proposition}
\begin{proof}
Assume $(X,d)$ is pointwise quasi-symmetric and fix $x_0\in X$. Let us denote by $J$ the identity mapping ${id: (X, d) \to (X, d^s)}$, and let us first check that $J$ is continuous at $x_0$. We know there exist $\alpha \geq 0$ and $\delta>0$ such that, if $d(x_0, x)<\delta$ then $d(x, x_0)\leq \alpha \, d(x_0, x)$. If $(x_n)$ is a sequence in $(X, d)$ converging to $x_0$ then $d(x_0, x_n)$ converges to $0$ and so $\bar{d} (x_0, x_n) =d(x_n, x_0)$ also converges to $0$. Therefore, $d^s(x_0, x_n)$ converges to $0$. On the other hand, if $d(x_0, x)<\delta$ then $d^s(x_0, x)\leq \max\{1,  \alpha\} \, d(x_0, x)$. Therefore $\SLip  J (x_0)\leq \max\{1,  \alpha\}$,  and so $J: (X, d) \to (X, d^s)$ is pointwise semi-Lipschitz. On the other hand, it is clear that $\SLip J^{-1}(x_0) \leq 1$. In this way we  obtain that $id: (X, d) \to (X, d^s)$ is a $C_{SL}$-homeomorphism. The converse is clear.
\end{proof}

\begin{remark}\em
Note that, if the quasi-metric spaces  $(X, d_X)$ and $(Y, d_Y)$  are pointwise quasi-symmetric, every pointwise semi-Lipschitz function $f: (X, d_X) \to (Y, d_Y)$ is continuous for the symmetrized metrics $d_X^s$ and $d_Y^s$.
\end{remark}

\begin{corollary}
Suppose that the quasi-metric space $(X, d)$ is pointwise quasi-symmetric. If $f:(X, d^s) \to (\R, \vert \cdot \vert)$ is Lipschitz, then $f:(X, d) \to (\R, d_u)$ is pointwise semi-Lipschitz.
\end{corollary}

\begin{proof}
By Proposition \ref{propequividentity}, the identity map $id:(X, d)\to (X, d^s)$ is pointwise semi-Lipschitz, and by Lemma~\ref{lemmacompos}, we obtain that $f\circ id:(X, d) \to (\mathbb{R},|\cdot|)$ is also pointwise semi-Lipschitz. Furthermore, since  $\mathrm{SLip}_{d,d_u}(f)\leq \mathrm{SLip}_{d,d_u^s}(f)=\mathrm{SLip}_{d,d_{|\cdot|}}(f)$ (see Remark~\ref{remarkconstantes}), we conclude that $f:(X, d) \to (\R, d_u)$ is pointwise semi-Lipschitz.
\end{proof}

As a consequence, we easily obtain the following separation property:

\begin{corollary}[Separation property]\label{separan}
Suppose that the quasi-metric space $(X, d)$  is pointwise quasi-symmetric. If $A, B$ are subsets of $X$ such that $d^s(A, B)>0$, there exists a function $f\in C_{SL}(X, (\R, d_u))$ with $0\leq f \leq 1$ such that $f(A)=0$ and $f(B)=1$.
\end{corollary}
\begin{proof}
Choose $f(x)=\frac{d^s(x,A)}{d^s(x,A)+d^s(x,B)}$. 
\end{proof}

We also obtain the following scalarization property, which will be useful in what follows:

\begin{lemma}\label{escalar}
Suppose that the quasi-metric spaces  $(X, d_X)$ and $(Y, d_Y)$  are pointwise quasi-symmetric, and let $h:X \to Y$. The following statements are equivalent:
\begin{itemize}
\item[(a)] $h \in C_{SL}(X, Y)$
%\item[(b)] For every bounded non-negative function ${f\in C_{SL}(Y, (\mathbb{R},|\cdot|))}$, we have ${f\circ h \in C_{SL}(X, (\mathbb{R},|\cdot|))}$.
\item[(b)] ${f\circ h \in C_{SL}(X, (\mathbb{R},|\cdot|))}$ for every bounded non-negative function ${f\in C_{SL}(Y, (\mathbb{R},|\cdot|))}$.
%\item[(c)] For every bounded non-negative function $f\in C_{SL}(Y, (\mathbb{R}, d_u))$, we have $f\circ h \in C_{SL}(X, (\mathbb{R}, d_u))$.
\item[(c)] $f\circ h \in C_{SL}(X, (\mathbb{R}, d_u))$ for every bounded non-negative function $f\in C_{SL}(Y, (\mathbb{R}, d_u))$.
\end{itemize}
\end{lemma}

\begin{proof}
The implications (a) $\Rightarrow$ (b) and (a) $\Rightarrow$ (c) follow from Lemma \ref{lemmacompos}. Now let us prove (b) $\Rightarrow$ (a). It is only in this implication where we use the hypothesis of pointwise quasi-symmetry. Taking into account Corollary \ref{separan}, we know that bounded non-negative functions in $C_{SL}(Y, (\mathbb{R}, \vert \cdot \vert))$ separate points and closed sets in $(Y, d_Y^s)$, and therefore also in $(Y, d_Y)$. Thus, $(Y, d_Y)$ has the initial topology for this family of functions, which implies that $h$ is continuous.

For each point $q \in Y$, consider the function $f_q(y) = \min \{d_Y(q, y), 1\}$. We want to show that $f_q \in C_{SL}(Y, (\mathbb{R}, \vert \cdot \vert))$. First, observe that for all $y, y_0$ in $Y$, we have
$$
f_q(y) - f_q(y_0) \leq d_Y(y_0, y).
$$
Now fix $y_0\in Y$. Due to the pointwise quasi-symmetry of $(Y, d_Y)$, we know that there exist $\alpha \geq 0$ and $\delta > 0$ such that if $d_Y(y_0, y) < \delta$, then
$$
d_Y(y, y_0) \leq \alpha d_Y(y_0, y),
$$
and therefore
$$
f_q(y_0) - f_q(y) \leq d_Y(y, y_0) \leq \alpha d_Y(y_0, y),
$$
which implies
$$
\vert f_q(y) - f_q(y_0) \vert \leq \max \{1, \alpha \} d_Y(y_0, y).
$$
From this, we conclude that $\SLip(f_q)(y_0) \leq \max\{1, \alpha\} < +\infty$. As  a consequence, we know that $f\circ h \in C_{SL}(X, (\mathbb{R}, |\cdot|))$.

Now, let $x_0 \in X$ and consider $y_0 = h(x_0)$. Then, by the continuity of $h$, for $x$ close enough to $x_0$ we have that
$$
d_Y(h(x_0), h(x)) =  \min \{d_Y(h(x_0), h(x)), 1\} = \vert (f_{y_0} \circ h)(x_0) - (f_{y_0} \circ h)(x) \vert.
$$
Thus we obtain that $\SLip h(x_0) = \SLip (f_{y_0} \circ h)(x_0) < +\infty$.

Finally, the implication (c) $\Rightarrow$ (a) is analogous to the previous one. We just need to take into account that we also have that $f_{y_0} \in C_{SL}(Y, (\mathbb{R}, d_u))$ and that
%in the final part of the proof,
$$
d_Y(h(x_0), h(x)) =  \min \{d_Y(h(x_0), h(x)), 1\} = d_u \left((f_{y_0} \circ h)(x_0),  (f_{y_0} \circ h)(x)\right).
$$
\end{proof}
%%%%%%%%%%%%%%%%%%%%%%%%%%%%%%%%%%%%%%%%%%%%%%

Next we are going to consider the uniform version of pointwise quasi-symmetric spaces.

\begin{definition}\label{unifquasisym}\em
We will say that a quasi-metric space $(X, d)$ is {\em uniformly quasi-symmetric} if it is pointwise quasi-symmetric and
$$
\sup_{x \in X} \sigma(x) < +\infty.
$$
\end{definition}

We the have the following simple characterization.

%%%%%%%%%%%%%%%%%%%%%%%%%%%
\begin{proposition} Let $(X, d)$ be a quasi-metric space. The following statements are equivalent:
\begin{itemize}
\item[(a)] $(X,d)$ is uniformly quasi-symmetric.
\item[(b)] The identity $id: (X, d) \to (X, d^s)$ is a $D_{SL}$-homeomorphism.
\end{itemize}
\end{proposition}

\begin{proof}
Suppose that $(X, d)$ is uniformly  quasi-symmetric and let ${K:= \sup_{x \in X} \sigma(x) < +\infty}$. By Proposition \ref{propequividentity}, we know that the identity $J=id: (X, d) \to (X, d^s)$ is a  $C_{SL}$-homeomorphism. Furthermore, the same proof of Proposition~\ref{propequividentity} shows that ${\SLip (J)(x_0) \leq \max\{1, K\}}$ and ${\SLip (J^{-1})(x_0) \leq 1}$ for each point $x_0\in X$. Thus, we obtain (b). The converse is clear.
\end{proof}

 To finish this subsection, we will refine Lemma \ref{escalar} using the following local version of uniform quasi-symmetry.

\begin{definition} \em
We will say that a quasi-metric space $(X, d)$ is {\em locally quasi-symmetric} if for each point $x_0 \in X$, there exist a constant $K_{x_0} > 0$ and a neighborhood $U^{x_0}$ such that $\sigma(x) \leq K_{x_0}$ for all $x \in U^{x_0}$.
\end{definition}

\begin{remark}\em
A quasi-metric space $(X, d)$ is locally quasi-symmetric if and only if the identity $id: (X, d) \to (X, d^s)$ ``locally belongs'' to the space $D_{SL}((X, d), (X, d^s))$, i.e., each point $x_0 \in X$ has a neighborhood $U^{x_0}$ such that $id: (U^{x_0}, d) \to (U^{x_0}, d^s)$ belongs to the space $D_{SL}((U^{x_0}, d), (U^{x_0}, d^s))$.
\end{remark}

\begin{remark}\em
Of course, every metric space is uniformly quasi-symmetric. Furthermore,  every Finsler manifold is locally quasi-symmetric as a quasi-metric space (see e.g. \cite{BCS}). In particular, every compact Finsler manifold is uniformly quasi-symmetric.
\end{remark}

\begin{example}\em
Note that a uniformly quasi-symmetric quasi-metric space can have index of symmetry $0$. As a simple example, consider a sequence $(I_n)$ of disjoint open intervals in $\R$ and, on the union space $X=\cup_n I_n$, define the quasi-metric $d$ given by
$$
d(p,q)=\left\{\begin {array}{ll} \vert p-q\vert &\text{if}\,\,x,y\in I_n \,\, \text{for some} \,\, n \\[10pt]
1 &\text{if} \,\, p\in I_n, \, q\in I_m , \, \,  n< m \\[10pt]
n &\text{if} \,\, p\in I_n, \, q\in I_m , \, \,  n> m
\end{array}
\right.
$$
It is easily seen that, in this case, $\sigma (p)=1$ for every $p\in X$. On the other hand, if we choose $p_n\in I_n$ and $q\in I_q$, we have that
$$
c_d(X)=\inf_{d(p,q)>0}\frac{d(q,p)}{d(p,q)} \leq \inf_{m}\frac{d(q, p_n)}{d(p_n, q)}= \inf_{n}\frac{1}{n}=0.
$$
\end{example}

We can compare the notions of partial symmetry as follows, with all inclusions being strict.
%Sin bordes
%\begin{table}[H]
%\begin{tabular}{ccccccccccc}
%\begin{tabular}[c]{@{}c@{}}Quasi-\\ metric\\ spaces\end{tabular} & $\supsetneq$ & \begin{tabular}[c]{@{}c@{}}Pointwise\\ quasi-\\ symmetric \\ spaces\end{tabular} & $\supsetneq$ & \begin{tabular}[c]{@{}c@{}}Pointwise\\ quasi-\\ symmetric \\ spaces\end{tabular} & $\supsetneq$ & \begin{tabular}[c]{@{}c@{}}Locally\\ quasi-\\ symmetric\\ spaces\end{tabular} & $\supsetneq$ & \begin{tabular}[c]{@{}c@{}}Spaces with\\ positive \\ index of\\ symmetry\end{tabular} & $\supsetneq$ & \begin{tabular}[c]{@{}c@{}}Metric\\ spaces\end{tabular}
%\end{tabular}
%\end{table}

\begin{table}[H]
\begin{tabular}{|ccccccccccc|}
\hline
\begin{tabular}[c]{@{}c@{}}Quasi-\\ metric\\ spaces\end{tabular} & $\supsetneq$ & \begin{tabular}[c]{@{}c@{}}Pointwise\\ quasi-\\ symmetric \\ spaces\end{tabular} & $\supsetneq$ & \begin{tabular}[c]{@{}c@{}}Pointwise\\ quasi-\\ symmetric \\ spaces\end{tabular} & $\supsetneq$ & \begin{tabular}[c]{@{}c@{}}Locally\\ quasi-\\ symmetric\\ spaces\end{tabular} & $\supsetneq$ & \begin{tabular}[c]{@{}c@{}}Spaces with\\ positive \\ index of\\ symmetry\end{tabular} & $\supsetneq$ & \begin{tabular}[c]{@{}c@{}}Metric\\ spaces\end{tabular} \\ \hline
\end{tabular}
\end{table}
Now in the class of locally quasi-symmetric spaces, we can refine the result of Lemma \ref{escalar}:

\begin{lemma}\label{lemmaCSL}
Suppose that the quasi-metric spaces $(X, d_X)$ and $(Y, d_Y)$ are locally quasi-symmetric, and let $h:X \to Y$. The following statements are equivalent:
\begin{itemize}
\item[(a)] $h \in C_{SL}(X, Y)$.
\item[(b)] For every bounded non-negative function $f\in D_{SL}(Y)$, we have  $f \circ h \in C_{SL}(X)$.
\item[(c)] For every bounded non-negative function $f\in D_{SL}(Y, (\mathbb{R}, d_u))$, we have $f \circ h \in C_{SL}(X, (\mathbb{R}, d_u))$.
\end{itemize}
\end{lemma}

\begin{proof}
Let us show that (b)$ \Rightarrow$ (a). First, fix $y_0 \in Y$. By hypothesis, there exist $K > 0$ and $\delta > 0$ such that if $d_Y(y_0, y) < \delta$, then $\sigma(y) < K$. Now choose $0 < r < \delta$ and define
$$
g_{y_0}(y) = \min \{d_Y(y_0, y), r\}.
$$
As in the proof of Lemma \ref{escalar}, we have $g_{y_0} \in C_{SL}(Y, (\mathbb{R}, |\cdot|))$. Moreover, we observe the following:

- If $d_Y(y_0, y) < \delta$, then $\SLip g_{y_0}(y) \leq \max\{1, K\}$.

- If $d_Y(y_0, y) > r$, then $g_{y_0}$ is constant with value $r$ in a neighborhood of $y$, and thus $\SLip g_{y_0}(y) = 0$.

This shows that $g_{y_0} \in D_{SL}(Y, (\mathbb{R}, |\cdot|))$. From here, we deduce that bounded non-negative functions in $D_{SL}(Y, (\mathbb{R}, |\cdot|))$ separate points and closed sets of $(Y, d_Y)$. Therefore, $(Y, d_Y)$ has the initial topology with respect to this family of functions, and hence the function $h$ is continuous. In addition, we have that $g_{y_0} \circ h \in D_{SL}(X, (\mathbb{R}, |\cdot|))$.

Now, let $x_0 \in X$ and consider $y_0 = h(x_0)$. Taking into account that, for $x$ close enough to $x_0$,
$$
d_Y(h(x_0), h(x))  = \min \{ d_Y(h(x_0), h(x)), r \}  = |(g_{y_0} \circ h)(x_0) - (g_{y_0} \circ h)(x)|,
$$
we obtain that $\SLip h(x_0) = \SLip (g_{y_0} \circ h)(x_0) <+\infty$. In this way,  $h \in C_{SL}(X, Y)$.

On the other hand, the implication (c) $\Rightarrow$ (a) is analogous.

\end{proof}
%%%%%%%%%%%%%%%%%%%%%%%%%%%%%%%%%%%%%%%%%%%%%

\subsection{Banach-Stone Theorems for pointwise semi-Lipschitz functions}

Our goal now is to offer some form of Banach-Stone type results characterizing the pointwise semi-Lipschitz structure of a quasi-metric space $(X, d_X)$ in terms of the topological-algebraic structure of certain function spaces of real-valued pointwise semi-Lipschitz functions defined on $X$. However, the spaces $C_{SL}(X)$ and $D_{SL}(X)$ do not possess a convenient algebraic structure. Indeed, it is clear that, in general,  the product of two semi-Lipschitz real valued functions on $X$ is not semi-Lipschitz, although it is so if the functions are {\em bounded and non-negative}. So, we are lead to consider bounded non-negative functions in order to guarantee the stability under product in our function spaces. Furthermore, $C_{SL}(X)$ and $D_{SL}(X)$ are not in general linear spaces, as we have seen in Example \ref{funcontDSL}. Taking all this into account, we proceed as follows. We first define $C_{SL}^{b,+}(X)$ (and respectively, $D_{SL}^{b,+}(X)$) the space of bounded and non-negative functions in $C_{SL}(X)$ (respectively, in $D_{SL}(X)$). These spaces are {\em cones}, and are also closed under multiplication. Then we define
\begin{itemize}
\item[$\diamond$] $\mathcal{A}_{CSL}(X)$ as the linear span of $C_{SL}^{b,+}(X)$,
\item[$\diamond$] and $\mathcal{A}_{DSL}(X)$ as the linear span of $D_{SL}^{b,+}(X)$.
\end{itemize}

Then it is clear that $\mathcal{A}_{CSL}(X)$ and $\mathcal{A}_{DSL}(X)$ are {\em algebras} of functions on $X$, which are {\em ordered} under the natural pointwise order of functions. Furthermore, they can be endowed with  natural {\em extended asymmetric norms} as per Definition \ref{defasymmnorm} in the following way (see \ref{conicsemiringtoalgebra} for details). We consider on $\mathcal{A}_{CSL}(X)$ the extended norm
$$
\| f |_C=\begin{cases}
\begin{array}{ccc}
&\| f \|_{\infty} &\text{ if } f\in C_{SL}^{b,+}(X) \\
&+\infty&\text{ if } f \notin C_{SL}^{b,+}(X).
\end{array}
\end{cases}
$$
On the other hand, we consider on $D_{SL}^{b,+}(X)$ the extended norm
$$
\| f |_D=\begin{cases}
\begin{array}{ccc}
&\max \{ \Vert f \Vert_{\infty},  \sup_{x\in X} \mathrm{SLip} f(x) \} &\text{ if } f\in D_{SL}^{b,+}(X) \\
&+\infty&\text{ if } f \notin D_{SL}^{b,+}(X).
\end{array}
\end{cases}
$$
Then, as a consequence of Theorem \ref{BSt}  in the next Section, we obtain the following Banach-Stone type results. Here  completeness of the spaces is needed. Recall that  quasi-metric space $(X,d_X)$ is said to be {\em bicomplete} if the symmetrized space $(X, d_X^s)$ is complete in the usual sense.

\begin{theorem} \label{csl-BS}
Suppose that the quasi-metric spaces $(X,d_X)$ and $(Y,d_Y)$ are bicomplete and locally quasi-symmetric. The following assertions are equivalent:
\begin{itemize}
\item[(a)] $X$ and $Y$ are $C_{SL}$-homeomorphic
\item[(b)] {\small $\mathcal{A}_{CSL}(X)$} and {\small $\mathcal{A}_{CSL}(Y)$} are isomorphic as ordered, extended asymmetric normed algebras.
\end{itemize}
\end{theorem}

\begin{proof}
Suppose first that there exists a $C_{SL}$-homeomorphism $\tau : (X, d_X) \to (Y,d_Y)$. It is clear from Lemma \ref{lemmacompos} that the corresponding composition operator $T:\mathcal{A}_{CSL}(Y) \to \mathcal{A}_{CSL}(X)$ given by $Tf:= f\circ \tau$ is well-defined, and it is easy to check that $T$ is an order-preserving isomorphism of extended asymmetric normed algebras (see Definition \ref{extasymiso}). 

For the converse, we apply Theorem \ref{BSt} to the symmetrized spaces $(X, d_X^s)$ and $(Y, d_Y^s)$. Here we choose the function space $\mathcal{F}(X,Y)=C_{SL}(X,Y)$, and the cones $\mathcal{G}(X)=C_{SL}^{b,+}(X)$ and $\mathcal{G}(Y)=C_{SL}^{b,+}(Y)$, both endowed with the corresponding norm $\| \cdot |_C$. The required scalarization properties (ii) and (ii') of Theorem \ref{BSt} follow from Lemma \ref{escalar},  and the separation property (iii) follows from Lemma \ref{separan}.
\end{proof}

In order to deal with $D_{SL}$-functions, we need the following variant of of Lemma \ref{lemmaCSL}.

\begin{lemma}\label{lemmaDSL}
Suppose that the quasi-metric spaces $(X, d_X)$ and $(Y, d_Y)$ are uniformly quasi-symmetric, and let $h:X \to Y$. The following assertions are equivalent:
\begin{itemize}
\item[(a)] $h \in D_{SL}(X, Y)$.
\item[(b)] There exists $C\geq 0$ such that, for every bounded non-negative function ${f\in D_{SL}(Y, (\mathbb{R}, |\cdot|))}$, we have  $f \circ h \in D_{SL}(X, (\mathbb{R}, |\cdot|))$ and $\| f\circ h |_D \leq C \, \| f |_D$.
\item[(c)] There exists $C\geq 0$ such that, for every bounded non-negative function $f\in D_{SL}(Y, (\mathbb{R}, d_u))$, we have $f \circ h \in D_{SL}(X, (\mathbb{R}, d_u))$ $\| f\circ h |_D \leq C \, \| f |_D$.
\end{itemize}
\end{lemma}

\begin{proof}
The implications (a) $\Rightarrow$ (b) and (a)$ \Rightarrow$ (c) follow from Lemma \ref{lemmacompos}. Suppose now that (b) holds. By Lemma \ref{lemmaCSL} we obtain that $h \in C_{SL}(X, Y)$. If we denote ${K:= \sup_{x \in X} \sigma(x) < +\infty}$ we obtain, as in the proof of Lemma \ref{escalar}, that for each $y_0\in Y$, the function ${f_{y_0}(y) = \min \{d_Y(y_0, y), 1\}}$ satisfies that $\SLip(f_{y_0}(y_0)) \leq \max\{1, K\}$, and therefore  $\| f_{y_0} |_D\leq \max\{1, K\}$. Now for each $x_0\in X$ we choose $y_0=h(x_0)$ and we obtain, again as in the proof of Lemma \ref{escalar},  that
$$
\SLip h(x_0) = \SLip (f_{y_0} \circ h)(x_0) \leq \sup_{x\in X} \mathrm{SLip} (f_{y_0} \circ h)(x) \leq \| f_{y_0} \circ h |_D \leq C \, \| f_{y_0} |_D
\leq C \, \max\{1, K\}.
$$
In this way we see that $h \in D_{SL}(X, Y)$.

The proof of (c) implies (a) is analogous.
\end{proof}

Finally, as a consequence of Theorem \ref{BSt}, we obtain the following result. 
\begin{theorem} \label{dsl-BS}
Suppose that the quasi-metric spaces $(X,d_X)$ and $(Y,d_Y)$ are bicomplete and uniformly quasi-symmetric. The following assertions are equivalent:
\begin{itemize}
\item[(a)] $X$ and $Y$ are $D_{SL}$-homeomorphic.
\item[(b)] {\small $\mathcal{A}_{DSL}(X)$} and {\small $\mathcal{A}_{DSL}(Y)$}  are isomorphic as ordered, extended asymmetric normed algebras.
\end{itemize}
\end{theorem}
\begin{proof}
Suppose first that there exists a $D_{SL}$-homeomorphism $\tau : (X, d_X) \to (Y,d_Y)$. We see from Lemma \ref{lemmacompos} that the composition operator $T:\mathcal{A}_{DSL}(Y) \to \mathcal{A}_{DSL}(X)$ given by $Tf:= f\circ \tau$ is well-defined, and also in this case it is easy to check that $T$ is an order-preserving isomorphism of extended asymmetric normed algebras (see Definition \ref{extasymiso}). 

For the converse, we apply again Theorem \ref{BSt} to the symmetrized spaces $(X, d_X^s)$ and $(Y, d_Y^s)$. We choose the function space $\mathcal{F}(X,Y)=C_{SL}(X,Y)$, and the cones $\mathcal{G}(X)=D_{SL}^{b,+}(X)$ and $\mathcal{G}(Y)=D_{SL}^{b,+}(Y)$, both endowed with the corresponding asymmetric norm norm $\| \cdot |_D$. The required scalarization properties (ii) and (ii') of Theorem \ref{BSt} follow from Lemma \ref{lemmaDSL},  and the separation property (iii) follows from Lemma \ref{separan}. In this way we obtain a $C_{SL}$-homeomorphism $\tau : X \to Y$ such that
$$
Tf=f\circ \tau
$$
for all $f\in \mathcal{A}_{DSL}(Y)$. From the continuity of $T$, there exists a constant $C\geq 0$ such that
$\| f\circ \tau |_D = \| Tf |_D \leq C \, \| f |_D$, for all $f\in \mathcal{A}_{DSL}(Y)$. Thus from Lemma \ref{lemmaDSL} we obtain that $\tau \in D_{SL}(X,Y)$, and the same holds for $\tau^{-1}$.
\end{proof}

\section{An abstract Banach-Stone type theorem}\label{secbsabstracto}

\subsection{Preliminaries on normed conic-semirings and asymmetric extended normed algebras}
We start this subsection by reviewing several definitions and notions of asymmetric nature found in the literature, that will be necessary to state and prove the main results of this Section. For a more in-depth introduction to these concepts, we refer the reader to \cite{cobzas}, \cite{DSV}, \cite{V} and \cite{tesisFV}.

\begin{definition}\label{defasymmnorm}\em
Given a real vector space $E$, we denote by  $\|\cdot|:E\to\R_+=[0, +\infty)$ an {\em asymmetric norm} on $E$, that is, a function satisfying:
\begin{itemize}
\item[(i)] $\forall\, x,y\in E$:\, $\|x+y|\,\leq\,\|x|\,+\,\|y|$;  \smallskip
\item[(ii)] $\forall\, x\in E$:\quad $x=0 \iff \|x|=0$;  \smallskip
\item[(iii)] $\forall \,x\in E,\,\forall r > 0$:\, $\|r\,x|\,=r\,\|x|$.
\end{itemize}
%If we replace the second condition by
%$$
%{\rm (ii)' }\quad x=0\quad\iff\quad \begin{cases}\
%\|x|=0\\[0.15cm]
%\ \|\!-\!\!x|=0
%\end{cases}
%$$
%then we say that $\|\cdot|:E\to\R_+$ is an {\em asymmetric hemi-norm} on $E$. {\color{red} Jesús: hacen falta las dos nociones??}
\end{definition}

The terminology of \emph{asymmetric normed space} refers to pairs $(E,\|\cdot|)$, where $\|\cdot|$ is an asymmetric norm on $E$. The symbol $\|\cdot\vert$, using two vertical bars on the left and only one bar on the right side, serves as a reminder of the asymmetric nature of these type of functionals, in the sense that the values $\|x\vert$ and $\| -x\vert$ may not coincide.

 We may also consider, keeping the same notation, {\em extended asymmetric norms}, allowing $\|\cdot|$ to take the value $+\infty$. The notion of (symmetric) extended norms was studied by Beer and Vanderwerff in \cite{BV}, and was generalized shortly after by Salas and Tapia-Garc\'ia in \cite{ST}, where they studied extended semi-norms and extended locally convex spaces. Asymmetric norms with infinite values have been considered in \cite{V}. The same considerations can be applied to quasi-metric spaces, by allowing the quasi-metric to take the value $+\infty$, in which case we will refer to them as {\em extended quasi-metrics.}

 The notion of asymmetric norm can be generalized further by allowing asymmetries in an algebraic sense, which will be done by considering cones instead of linear spaces

\begin{definition}[Cone] \label{defcone}\em
 A subset $C$ of a real linear space $E$ will be called a \emph{cone} if it is closed under finite sums and under multiplication by non-negative scalars. In other words, for all $x,y\in C$ and $r\geq 0$:
\begin{enumerate}
\renewcommand\labelenumi{(\roman{enumi})}
\leftskip .35pc
\item $x+y\in C$,
\item $r\cdot x\in C$.
\end{enumerate}
In particular, under this definition, every cone is a convex subset of $E$ containing the origin. A \emph{subcone} of a cone $C$ will be any set that is a cone itself, and a \emph{linear map} on a cone $C$ will be the restriction of a linear map (with values on some linear space $F$) on the linear space $\mathrm{span}(C)=C-C\subseteq E$.
\end{definition}

We will consider cones endowed with an asymmetric norm, as follows.

\begin{definition}[Conic norm] \label{defnorm}\em
A {\em conic-norm} on a cone $C$ is a function $\|\cdot|\hbox{\rm :}\ C\to \mathbb{R}_{+}$ such that for all
$x,y\in C$ and $r\geq 0$:\smallskip
\begin{enumerate}
\renewcommand\labelenumi{(\roman{enumi})}
\leftskip .35pc
\item $\|x+y|\leq \|x|+\|y|$,  for all $x,y \in C$ \smallskip
\item $\|x|=0\ \iff\ x=0$ \smallskip
\item $\|r\cdot x|=r\|x|$,  for all $x\in C$ and $r \geq 0$. \vspace{-.4pc}\smallskip
\end{enumerate}
\medskip
The pair $(C,\|\cdot|)$ is called \emph{normed cone}. %If we replace condition (ii) by
%$$
%{\rm (ii)' }\qquad  x=0\iff \forall z\in C,\, \left[ x+z=0 \implies \|x|=\|z|=0\right],
%$$
%then we say that $\|\cdot|\hbox{\rm :}\ C\to \mathbb{R}_{+}$ is a {\em conic hemi-norm}. A cone equipped with either a conic-norm or a conic hemi-norm will be called {\em normed cone}.

\begin{remark}\em
Asymmetric normed spaces (as per Definition \ref{defasymmnorm}) are a particular case of normed cones.
\end{remark}
\end{definition}

Normed cones and asymmetric normed spaces can be regarded as {\em extended} quasi-metric spaces.
\begin{proposition}
Let $C$ be a normed cone. Then, the function
\[ d_e(x,y)=
\begin{cases}\
\|y-x|\quad \text{ if }y-x\in C\\[0.15cm]
\ +\infty\qquad\text{ otherwise}
\end{cases}
\]
\noindent defines an extended quasi-metric on $C$. Moreover, $d_e$ is a quasi-metric (with finite values) if and only if $C$ is a linear space.
\end{proposition}
In what follows, unless stated otherwise, all topological notions on a normed cone $(C,\|\cdot|)$ will be with respect to the forward topology of the extended quasi-metric induced by $\|\cdot|$. The only space exempt from this convention will be $\mathbb{R}$, which will be assumed to carry its usual topology and metric.

For a proof of the following result we refer to Section 2.3.3 in \cite{cobzas}.

\begin{proposition}[Linear functionals over a normed cone]\label{caract}\em
Let $(C,\|\cdot|)$ be a normed cone and $\varphi:C\to \mathbb{R}$ a linear functional. Then the following are equivalent:
\begin{enumerate}
    \item[(i)] $\varphi$ is upper semicontinuous (in short, usc) from $(C,\|\cdot|)$ to $(\R, \vert \cdot \vert)$; \smallskip
    \item[(ii)] $\varphi$ is semi-Lipschitz  from $(C,\|\cdot|)$ to $(\R, d_u)$ (in short, $\varphi \in \mathrm{SLIP}(C))$;\smallskip
    \item[(iii)] there exists $M\geq 0$ such that $\varphi(x)\leq M\|x|$, for all $x\in C$.
\end{enumerate}
\end{proposition}

Proposition \ref{caract} will be helpful, as it allows us to verify the upper semicontinuity of linear functionals using a simple inequality. Keeping this in mind, the dual of a normed cone is defined as follows.

\begin{definition}[Dual normed cone]\label{dual cone}\em
Let $(C,\|\cdot|)$ be a normed cone. We define the \emph{dual cone} of $C$ as
$$C^*:=\{\varphi:C\to \mathbb{R}:\,\varphi\text{  is linear and usc}\}\,=\,\{\varphi\in\mathrm{SLIP}(C):\,\varphi \text{ is linear}\}.$$
For every $\varphi\in C^*$, the \emph{dual conic-norm} is defined by
\begin{equation*}\label{dual-conic}
\|\varphi|^*:=\sup_{\|x|\leq 1} \max\{\varphi(x),0\}=\sup_{\|x|\leq 1} \varphi(x).
\end{equation*}
\end{definition}

It is easy to check that $\|\cdot|^*$ is a conic-norm on $C^*$, and that it coincides with $\mathrm{SLIP}(\cdot)$. Whenever we refer to the dual of an (extended) asymmetric normed space, it will be assumed that we are talking about dual cones in the sense of Definition \ref{dual cone}. It should be noted that infinite dimensional asymmetric normed spaces often fail to have linear duals (see \cite{BF}.)

%\begin{remark}\em
%It is important to note that even though the linear functions $\varphi \in C^*$ are restrictions of linear functions defined in $\mathrm{span}(C)$, upper semi-continuity is only required to hold on the cone $C$ itself. As a consequence, even if an ambient linear space $E$ containing $C$ is endowed with an asymmetric norm whose restriction coincides with the conic norm of $C$, an element of $C^*$ may fail to be upper semi-continuous in $E$. {\color{red} Jesús:tenemos algún ejemplo de esto?? }{\color{blue} Francisco: Tengo un ejemplo, pero es en un cono con hemi-norma (que no distingue el $0$): $C=\R_+$, $E=\R$, ambos con la hemi-norma $u$. La función $f(x)=-x$ es lineal sobre $\R$, pero no es usc (no es acotada sobre los números negativos, que están en la bola unitaria de $(\R,d_u)$), pero su restrucción a $\R_+$ tiene norma dual $0$. No recuerdo que tenía en mente cuando escribí este Remark, supongo que podríamos quitarlo.}
%\end{remark}

%\begin{definition}[Asymmetric weak topologies]\label{weakstar}
%Let $E$ be an (extended) asymmetric normed space with dual cone $E^*$. The \emph{weak-star} topology $w^*$ on $E^*$ is defined as the coarsest topology that makes every evaluation functional $\{\widehat{x}:E^*\to (\mathbb{R},|\cdot|),\:x\in E\}$ continuous.
%\end{definition}

Next, we recall the algebraic definition of a semiring: a {\em semiring} is a commutative monoid endowed with a compatible multiplication operation that distributes over the addition of the monoid.

\begin{definition}[Conic-semiring \cite{V}]\label{conicsemiring}\em
 A \em conic-semiring \em is a cone (as per Definition~\ref{defcone}) endowed with a multiplication that makes it a semiring. If the cone is endowed with a conic-norm for which there exists a constant $K\geq 0$ such that $\|fg|\leq K\|f|\|g|$ for all $f,g$ in the cone, we will call it a normed conic-semiring. A normed conic-semiring will be called unital if it has a multiplicative unit.
\end{definition}

Using this definition it is easy to check that, for any quasi-metric space $X$, the spaces $C_{SL}^{b,+}(X)$ and $D_{SL}^{b,+}(X)$ considered in Section 1.1.4 are  normed conic-semirings when endowed with their natural operations and conic-norms.

Just like the notion of cones is an asymmetric version of real linear spaces (where the additive group is reduced to a monoid and the scalars are restricted to $\mathbb{R}_+$), the notion of semiring can be seen as the asymmetric version of rings (replacing the additive group with a monoid). By combining these two ideas, we can view conic-semirings as an asymmetric version of algebras, sacrificing part of the algebraic structure of algebras in order to maintain a well defined norm. A similar result can be achieved using the opposite approach, that is, by weakening the properties of the norm in order to preserve the linear structure of algebras. %Recall that an \textbf{extended} asymmetric norm has the same properties as an asymmetric hemi-norm (see Definition~\ref{defasymmnorm}), but is allowed to take the value $+\infty$.

\begin{definition}[Finite subcone]\label{finitesubcone}\em
Let $(E,\|\cdot|)$ be an extended asymmetric normed space. The subset $F=\{ x\in E ~:~ \|x|<+\infty\}$ (which is always a cone) is called the {\em finite subcone} of $E$.
\end{definition}

\begin{definition}[Extended asymmetric normed algebra \cite{V}]\label{extasymalg}\em
 An algebra $\mathcal{A}$ endowed with an extended asymmetric norm $\|\cdot |$ will be called an \textit{extended asymmetric normed algebra} if the finite subcone satisfies a submultiplicative condition for the norm, i.e., there exists $K\geq 0$ such that $\|fg|\leq K\|f|\|g|$ for all $f,g\in \mathcal{A}$ such that $\|f|,\|g|<+\infty$.
\end{definition}

Extended asymmetric normed algebras and normed conic-semirings are closely connected, as we can see in the following simple result.

\begin{proposition}
    Let $(\mathcal{A},\|\cdot|)$ be an extended asymmetric normed algebra. Then, the finite subcone of $\mathcal{A}$ is a normed conic-semiring when endowed with the norm of $\mathcal{A}$.
\end{proposition}

Under Definition~\ref{conicsemiring}, we can also also define an extended asymmetric normed algebra from a given normed conic-semiring, by simply extending the conic norm in a trivial way to the span of the cone, as the next proposition shows.

\begin{proposition}\label{conicsemiringtoalgebra}
    Let $(C,\|\cdot|)$ be a normed conic-semiring. Set $\mathcal{A}=\mathrm{span}(C)$, and for any $a\in \mathcal{A}$, define
$$
\| a |_{\mathcal{A}}=\begin{cases}
\begin{array}{ccc}
&\| a| &\text{ if }a\in C \\
&+\infty&\text{ if }a\notin C
\end{array}
\end{cases}
$$
Then, $(\mathcal{A}, \| \cdot |_{\mathcal{A}})$ is an extended asymmetric normed algebra, and the finite subcone of $\mathcal{A}$ coincides with $C$. In this case, we will say that the extended asymmetric normed algebra $\mathcal{A}$ is \textit{generated} by the normed conic-semiring $C$.
\end{proposition}
\begin{proof}
Let us verify that $\mathcal{A}$ is an algebra. Since it is by definition a linear space, we only need to check that it is closed under multiplication. Let $x,y\in \mathcal{A}$. Since $C$ is a cone, $\mathrm{span}(C)$ can be written as $C-C=\{c_1 - c_2 ~:~ c_1,c_2\in C\}$, so we can write
$$xy= (x_1-x_2)(y_1-y_2)=x_1y_1-x_1y_2-x_2y_1+x_2y_2=(x_1y_1+x_2y_2) - (x_1y_2 +x_2y_1),$$
\noindent with $x_i$ and $y_i$ in $C$, for $i=\{1,2\}$. It follows that $xy\in C-C=\mathcal{A}$. The remaining properties hold by definition.
\end{proof}

It is worth mentioning that the notions of cones and conic-semirings can be formulated in an abstract way, which does not require for the cones to be contained in a linear space. This more general framework will not be needed here, as the main result of this Section deals with cones of real valued functions defined on a given quasi-metric space $X$, which are always contained in the linear space $\mathbb{R}^X$. For more details on the abstract version of these notions, we refer the reader to Section 2.1.3 and Section 5.1.2 of \cite{tesisFV}, as well as to Section 3.1 of \cite{V}.

\begin{definition}\label{continuityextnorm}\em
A linear map $T$ between two extended asymmetric normed spaces $(\mathcal{A}_2,\|\cdot|_2)$ and $(\mathcal{A}_1,\|\cdot|_1)$ is said to be {\em continuous} if there exists a constant $K\geq 0$ such that $$\|Tf|_1\leq K\|f|_2,$$
\noindent for all $f\in \mathcal{A}_2$. The least constant $K$ satisfying the inequality above is called the norm of $T$, denoted by $\|T|$.
\end{definition}

Note that the above definition is equivalent to the usual forward-forward continuity for the respective asymmetric extended norms (see Propositions 2.2.1 and 2.2.2 of \cite{cobzas} for the proof in the case of asymmetric normed spaces, which can also be applied to extended norms).

Note also that a continuous linear map necessarily sends the finite subcone of its domain into the finite subcone of its range.

\begin{definition}\label{extasymiso}\em
Given two extended asymmetric normed algebras $(\mathcal{A}_1,\|\cdot|_1)$ and $(\mathcal{A}_2,\|\cdot|_2)$, a mapping $T:\mathcal{A}_2\to \mathcal{A}_1$ is called an \textit{extended asymmetric normed algebra isomorphism} provided:
\begin{enumerate}
    \item[(i)] $T$ is linear and bijective,
    \item[(ii)] $T$ is bicontinuous, i.e., $T$ and $T^{-1}$ are continuous in the sense of Definition~\ref{continuityextnorm}.
    \item[(iii)] $T(fg)=Tf\cdot Tg$ for all $f,g\in \mathcal{A}_2.$
\end{enumerate}
The isomorphism $T$ is called an extended asymmetric normed algebra \textit{isometry} if $$\|Tf|_1=\|f|_2$$
\noindent for all $f\in \mathcal{A}_2$, or equivalently, if $\|T|=\|T^{-1}|=1$.
\end{definition}

We will also need the following definition.

\begin{definition}[Positive isomorphism]\em
Let $\mathcal{A}(X)$ and $\mathcal{A}(Y)$ be extended asymmetric normed algebras of real-valued functions over sets $X$ and $Y$, respectively. A mapping $T:\mathcal{A}(Y)\to \mathcal{A}(X)$ is called {\em positive} if  $Tf\geq 0$ whenever $f\geq 0$. In the case that $T$ is an isomorphsm, we say that $T$ is {\em order preserving} if $T$ and $T^{-1}$ are both positive.
\end{definition}

\subsection{Topological version}\label{subsecbstopo}

Before stating our main result, we fix some notation. If $X$ and $Y$ are topological spaces, $C(X, Y)$ denotes as usual the space of continuous mappings between $X$ and $Y$.  Furthermore, $C^{b,+}(X)$ will denote the space of bounded, non-negative, continuous real functions on $X$, where the real line $\R$ is endowed with its usual metric. 
\begin{theorem}\label{BSt}
Let $(X,d_X)$ and $(Y,d_Y)$ be complete metric spaces, let $\mathcal{F}(X,Y)$ be a subset of $C(X,Y)$, and let $\mathcal{G}(X)$ and $\mathcal{G}(Y)$ be subcones of $C^{b,+}(X)$ and $C^{b,+}(Y)$, respectively, such that:
\begin{itemize}
     \item[(i)] $\mathcal{G}(X)$ and $\mathcal{G}(Y)$ are endowed with conic norms, which are finer than $\|\cdot\|_\infty$, and which make them into unital normed conic-semirings under the usual addition and multiplication of real-valued functions.
    \item[(ii)] $h\in\mathcal{F}(X,Y) $ provided $f\circ h\in \mathcal{G}(X)$ for all $f\in \mathcal{G}(Y)$.
    \item[(ii')] $h\in\mathcal{F}(Y,X) $ provided $f\circ h\in \mathcal{G}(Y)$ for all $f\in \mathcal{G}(X)$.
    \item[(iii)] For $Z\in \{X,Y\}$, $\mathcal{G}(Z)$ is uniformly separating for $(Z,d_Z)$, in the sense that, for every pair of subsets $A$ and $B$ of $Z$ with $d_Z(A,B)>0$, there exists some $f\in \mathcal{G}(Z)$ such that $\overline{f(A)}^{|\cdot|} \cap \overline{f(B)}^{|\cdot|} =\emptyset$.
\end{itemize}
\noindent For $Z\in\{X,Y\}$, denote $\mathcal{A}(Z)=\mathrm{span}(\mathcal{G}(Z))$, endowed with the extended asymmetric norm induced by $\mathcal{G}(Z)$ and its natural algebra structure. If $T:\mathcal{A}(Y)\to \mathcal{A}(X)$ is an order-preserving isomorphism of extended asymmetric normed algebras, there exists a bijection $\tau\in \mathcal{F}(X,Y)$ with $\tau^{-1}\in \mathcal{F}(Y,X)$, and such that
$$Tf=f\circ \tau$$
for all $f\in \mathcal{A}(Y)$.
\end{theorem}

\begin{remark}\em
In this result, the set $\mathcal{F}(X,Y)$ has the role of ``space of morphisms'' between $X$ and $Y$. The function spaces $\mathcal{G}(X)$ and $\mathcal{G}(Y)$ need not be of the ``same nature'' as $\mathcal{F}(X,Y)$, and this will often be the case in our examples (see forthcoming Corollary~\ref{examplesbst}).
\end{remark}

Let us begin with some preliminaries for the proof of Theorem~\ref{BSt}. Thanks to Proposition~\ref{conicsemiringtoalgebra}, we know that $\mathcal{A}(X)$ and $\mathcal{A}(Y)$ are extended asymmetric normed algebras. Therefore, we can readily define the \emph{structure space} by
$$
\mathcal{S}(X):=\{\varphi:\mathcal{A}(X)\to \mathbb{R}~:~\varphi\text{ is linear, multiplicative, continuous and positive}\}\subset \mathcal{A}(X)^*
$$

\begin{remark}\label{caractcont}\em
Every linear, multiplicative and upper semicontinuous functional $\varphi$ is actually continuous. To see this, it suffices to show that $-\varphi$ is upper semicontinuous, for which we can use Proposition \ref{caract}. We need to given a bound for $-\varphi(f)$ for any $f\in \mathcal{A}(X)$. By denoting the constant function of value $1$ as $\mathds{1}$, we have
$$-\varphi(f)=\varphi(-\mathds{1})\varphi(f)\leq \varphi(-\mathds{1})\|f|\|\varphi|^*\leq K \|f|.$$
As a consequence, an equivalent definition of the structure space $\mathcal{S}(X)$ could be given by requiring each functional to be usc instead of continuous.
\end{remark}

\begin{proposition}\label{deltainS}
The set of evaluation functionals $\delta(X)=\{\delta_x:\mathcal{A}(X)\to\R~:x\in X\}$ is contained in $\mathcal{S}(X)$.
\end{proposition}
\begin{proof}
As usual, the evaluation functional $\delta_x$ is defined by $\delta_x(f)=f(x)$. It is clear that every $\delta_x$ is linear and positive on $\mathcal{A}(X)$. Upper semi continuity is deduced using Proposition \ref{caract} and the fact that the extended asymmetric norm on $\mathcal{A}(X)$ is finer than the supremum norm and continuity follows from Remark~\ref{caractcont}.
\end{proof}

\begin{proposition}\label{homeo}
$(X,d_X)$ is homeomorphic to $(\delta(X),\tau_p)$, where $\tau_p$ denotes the trace of the product topology of $\mathbb{R}^{\mathcal{A}(X)}$.
\end{proposition}
\begin{proof}
We start by noting that, since $\mathcal{A}(X)$ separates points and closed sets of $X$, this proposition is in fact a well known result in general topology (see for instance Theorem 8.16 of \cite{willard}). Nevertheless, we include the proof for the sake of completeness. We start by proving that $\delta: X\to \delta(X)$ is open. Let $U\subseteq X$ be an open set, and fix a point $\delta_x\in \delta(U)$. Since $x$ does not belong to the closed set $U^c$, we can use hypothesis $(iii)$ of Theorem~\ref{BSt} to obtain a function $f\in \mathcal{G}(X)$ that separates $\{x\}$ from $U^c$, that is, $f(x)\notin \overline{f(U^c)}$, which implies the existence of $\varepsilon>0$ such that $B(f(x),\varepsilon)\cap f(U^c)=\emptyset$. Consider now the $\tau_p$-neighborhood of $0$ defined by the separating function $f$ and the radius $\varepsilon$ of the ball, that is, $W=\{ \delta_z \in \mathcal{S}(X):~|\delta_z (f)|<\varepsilon\}$. Then, the set $\delta_x + W$ is a $\tau_p$-neighborhood of $\delta_x$ contained in $\delta(U)$. On the other hand, continuity of the mapping $\delta$ follows directly from the fact that the functions in $\mathcal{G}(X)$ (and therefore $\mathcal{A}(X)$) are continuous.
\end{proof}

\begin{remark}\em
    It is worth noting that an asymmetric normed space $E$ is not in general a topological vector space, as multiplication by scalars may fail to be continuous at points of the form $(0,x)\in \R\times E$. Fortunately, addition remains continuous, and therefore, translations of open sets are still open.
\end{remark}

\begin{proposition}\label{denso}
$\delta(X)$ is $\tau_p$-dense in $\mathcal{S}(X)$.
\end{proposition}
\begin{proof}
Consider a basic $\tau_p$-neighborhood of a function $\varphi\in \mathcal{S}(X)$ and $f_1,f_2,\cdots,f_n\in \mathcal{A}(X)$:
$$W=\{\psi \in \mathcal{S}(X):~|\psi(f_i)-\varphi(f_i)|<\varepsilon,~i=1,...,n\}.$$
Assume that $W\cap \delta(X)=\emptyset$, and consider the function  $g=\sum_{i=1}^n (f_i-\varphi(f_i))^2\in \mathcal{A}(X)$. Then, $g(x)\geq n\varepsilon^2>0$ and $\varphi(g)=0$, which contradicts the positivity of $\varphi$. Therefore, $W\cap\delta(X)\neq\emptyset$ and $\delta(X)$ is  $\tau_p$-dense in $\mathcal{S}(X)$.
\end{proof}

\begin{lemma}\label{superlema}
The following are equivalent:
\begin{itemize}
    \item[$(a)$] $\varphi\in \mathcal{S}(X)$ has a countable neighborhood basis.
    \item[$(b)$] There exists $x\in X$ such that $\varphi=\delta_x$
\end{itemize}
\end{lemma}
\begin{proof}
Assume $(b)$, and consider a countable neighborhood basis $(V_n)$ for $x\in X$ such that $\varphi=\delta_X$. Proposition~\ref{homeo} implies that the family $(\overline{\delta(V_n)}^{\tau_p})$ is a $\tau_p$-neighborhood basis for $\varphi$. Conversely, assume $\varphi\in \mathcal{S}(X)\setminus \delta(X)$ has such a neighborhood basis. By Proposition~\ref{denso}, there exist a sequence $(x_n)$ in $X$ such that $\delta_{x_n}$ converges to $\varphi$ in $(\mathcal{S}(X),\tau_p)$. This implies, by completeness of $(X,d)$, that $(x_n)$ has no Cauchy sub-sequence, otherwise such a sub-sequence would be convergent to $x\in X$, which would contradict the fact that $\varphi \notin \delta(X)$, as the product topology separates points. Therefore, there exists $\varepsilon>0$ and a sub-sequence $(x_{n_k})$ such that $d(x_{n_k},x_{n_j})\geq \varepsilon$ whenever $k\neq j$. Define $A=\{x_{n_k}\!~:~\!k \text{ is odd}\}$ and $B=\{x_{n_k}~\!:~\!k \text{ is even}\}$. Since $\mathcal{G}(X)$ is uniformly separating, we can find $f\in \mathcal{G}(X)$ such that $\overline{f(A)}\cap \overline{f(B)}=\emptyset$, but, since $(x_n)$ converges to $\varphi$ in $\tau_p$, we have that $f(x_n)$ converges (in the $|\cdot|$-topology of $\R$) to $\varphi(f)$, which then must belong to $\in \overline{f(A)}\cap \overline{f(B)}$.
\end{proof}

We can now proceed to the proof of Theorem~\ref{BSt}.

\begin{proof}
Let $T:\mathcal{A}(Y)\to \mathcal{A}(X)$ be an order-preserving isomorphism of extended asymmetric normed algebras, and consider the dual mapping $T^*:\mathcal{A}(X)^*\to \mathcal{A}(Y)^*$ defined by the formula
$$\langle T^*\varphi,f\rangle=\langle \varphi, Tf\rangle \text{ for all }f\in \mathcal{A}(Y).$$
We have that $T^*$ is well defined, linear, bijective and $\tau_p$-$\tau_p$ continuous. The algebraic properties of $T$, along with the assumed positivity, guarantee that $T^*$ sends positive and multiplicative functionals in $\mathcal{A}(X)$ to positive and multiplicative functionals in $\mathcal{A}(Y)$, that is, $T^*$ preserves the corresponding structure spaces. Moreover, Lemma~\ref{superlema} ensures that $T^*$ sends $\delta(X)$ into $\delta(Y)$. Then, we can define $\tau:X\to Y$ as $\tau(x)=\delta_Y^{-1}T^*(\delta_X(x))$, where $\delta_X$ and $\delta_Y$ are the corresponding embeddings of $X$ and $Y$ into $\mathcal{A}(X)$ and $\mathcal{A}(Y)$. The formula $Tf=f\circ \tau$ then follows from the definition of $\tau$. Injectivity of $\tau$ follows from the fact that $T$ is surjective and that $\mathcal{G}(X)$ is separating for $d_X$, and surjectivity follows directly from the properties of $T^*$. Finally, the continuity of $T$ implies that $T$ sends $\mathcal{G}(Y)$ into $\mathcal{G}(X)$, so hypothesis $(ii)$ guarantees that $\tau\in \mathcal{F}(X,Y)$. The same argument for the isomorphism $T^{-1}$ yields that $\tau^{-1}\in \mathcal{F}(Y,X)$.
\end{proof}

As we have already mentioned, Theorems \ref{csl-BS} and \ref{dsl-BS} were obtained as a consequence of Theorem \ref{BSt}. Some further applications are given below.

\begin{corollary}\label{examplesbst}
Theorem~\ref{BSt} can be applied to the following classes of spaces of real-valued functions $\mathcal{G}(X)$.
\begin{itemize}
    \item[(a)] $C^{b,+}(X)$ of bounded, non-negative, continuous functions,  on a completely {\em metrizable} topological space $X$, endowed with the supremum norm. In this case, $\tau$ will be an homeomorphism.
    \item[(b)] $D^{b,+}(X)$, of bounded, non-negative pointwise Lipschitz functions with bounded pointwise Lipschitz constant, on a complete metric space $X$, endowed with the norm
        $$\|f\|=\max\{\|f\|_\infty, \|\mathrm{Lip}(f)\|_\infty\}.$$
        In this case, $\tau$ will be a $C_{SL}$-homeomorphism.% bi-pointwise-Lipschitz homeomorphism.
    \item[(c)] $D^{b,+}_{SL}(X)$, of bounded, non-negative continuous functions with bounded ascendent metric slope, on a complete metric space $X$, endowed with the asymmetric norm
        $$
        \| f |_D=\max\{\|f\|_\infty, \|\mathrm{SLip} f\|_\infty\}.
        $$ In this case, $\tau$ will be a $C_{SL}$-homeomorphism.%a bi-pointwise-Lipschitz homeomorphism.
    %\item[$(e)$] $D_s^\infty(X)$, of bounded functions with bounded metric slope on \textbf{quasi-metric} spaces, obtaining a pointwise \textbf{semi}-Lipschitz homeomorphism.
%    \\
%    {\color{blue}Francisco: en los casos (d) y (e) no se como mostrar la escalarización. En los resultados de \cite{GJR} y \cite{V}, se comienza con un isomorfismo (continuo) de álgebras $T$, se obtiene un $h$ definido usando $T^*$, y luego se prueba que ese $h$ en particular es Lipschitz o semi-Lipschitz. Por el momento no se me ocurre como probar la propiedad de escalarización.  }
%    {\color{red}
%    \item[(d)] $C_{b,+}^1(\mathcal{X})$, of bounded, non-negative functions with bounded derivative on a connected, reversible and complete Finsler manifold $\mathcal{X}$, endowed with the norm $\|f\|=\max\{\|f\|_\infty, \|df\|_\infty\}$. In this case, $\tau$ will be a bi-Lipschitz diffeomorphism.
%    \item[(e)] $SC_{b,+}^1(\mathcal{X})$, of bounded, non-negative semi-Lipschitz functions of class $C^1$ on a connected and bicomplete Finsler manifold, endowed with the norm $\|f\|=\max\{\|f\|_\infty, \|df|_\infty\}$. In this case, $\tau$ will be a bi-Lipschitz diffeomorphism.}\\
%    
%    {\color{blue}
%    Propongo el siguiente cambio: usar como $\mathcal{F}(\mathcal{X},\mathcal{Y})$ el espacio $C_{SL}(\mathcal{X},\mathcal{Y})\cap C^1(\mathcal{X},\mathcal{Y})$. Agregaré una demostración para estas versiones $(d')$ y $(e')$.

 \item[(d)] $C_{b,+}^1(\mathcal{X})$, of bounded, non-negative functions with bounded derivative on a connected, reversible and complete Finsler manifold $\mathcal{X}$, endowed with the norm $\|f\|=\max\{\|f\|_\infty, \|df\|_\infty\}$. In this case, $\tau$ will be a $C^1$-smooth $C_{SL}$-homeomorphism.
    \item[(e)] $SC_{b,+}^1(\mathcal{X})$, of bounded, non-negative semi-Lipschitz functions of class $C^1$ on a connected and bicomplete Finsler manifold, endowed with the norm $\|f\|=\max\{\|f\|_\infty, \|df|_\infty\}$. In this case, $\tau$ will be a $C^1$-smooth $C_{SL}$-homeomorphism. We refer the reader to \cite{DJV} and \cite{V} for the definition of $\|df|_\infty$ on Finsler manifolds.

    \item[(f)] $\mathrm{LIP}_+^\infty(X)$, of bounded, non-negative Lipschitz functions on a complete quasi-convex %{\color{red} Jesús: OJO, en principio hace falta alguna condición para la escalarización, por ejemplo ésta (véase Theorem 3.12 in \cite{GJ}}
    metric space, endowed with the norm $\|f\|=\max\{\|f\|_\infty, \mathrm{LIP}(f)\}$. In this case, $\tau$ will be a bi-Lipschitz homeomorphism.
    %\item[$(i$)] $\mathrm{SLIP}^\infty(X)$, of bounded semi-Lipschitz functions on a \textbf{quasi-metric} space, obtaining a \textbf{semi}-Lipschitz homeomorphism.
    \item[(g)] $\mathrm{lip}_{+}(X)$ of non-negative locally flat Lipschitz functions on a compact and purely $1$-unrectifiable metric space $X$, endowed with the norm $\|f\|=\max\{\|f\|_\infty, \mathrm{LIP}(f)\}$. In this case, $\tau$ will be a bi-Lipschitz homeomorphism.
    %\item[(k)] $\mathrm{slip}^\infty(X)$ of bounded and locally flat semi-Lipschitz function on a bicomplete and forward\textbf{quasi}-metric space satisfying the uniform separation property, obtaining a semi-Lipschitz homeomorphism.
\end{itemize}
\end{corollary}
\begin{proof}
Case (a) is clear. In cases (b) and (c), we choose $\mathcal{F}(X,Y)=C_{SL}(X,Y)$, and the required scalarization properties  follow  from Lemma \ref{lemmaCSL}. Concerning case $(f)$, the  function space is $\mathcal{F}(X,Y)= \mathrm{LIP}(X,Y)$ and the scalarization property follows from Theorem 3.12 in \cite{GJ}. In all these cases, the separation property is clear. For case (e), we will apply Theorem~\ref{BSt} to the manifolds endowed with the symmetrized distances of the respective Finsler quasi-metrics, and using $\mathcal{F}(\mathcal{X},\mathcal{Y})=C_{SL}(\mathcal{X},\mathcal{Y})\cap C^1(\mathcal{X},\mathcal{Y})$. To show that condition (ii) holds, we start by proving that $h:\mathcal{X}\to \mathcal{Y}$ belongs to $C_{SL}(\mathcal{X},\mathcal{Y})$. To this end, fix $x_0\in \mathcal{X}$, and consider the semi-Lipschitz function $f(y)=d_\mathcal{Y}(h(x_0),y)\wedge 1$. For any $\varepsilon>0$, we can take a smooth semi-Lipschitz approximation $\tilde{f}$ of $f$ (see Corollary 2.31 in \cite{DJV}) such that $\vert f(y)-\tilde{f}(y)\vert <\varepsilon$ for all $y\in \mathcal{Y}$. Upon translation by a constant, we may assume that $\tilde{f}$ is non negative, so that $\tilde{f}\in SC_{b,+}^1(\mathcal{Y})$. Using the hypothesis of (ii), we obtain that $\tilde{f}\circ h\in SC_{b,+}^1(\mathcal{X})$. In particular, $\tilde{f}\circ h$ is semi-Lipschitz. Let $L=\mathrm{SLIP}(\tilde{f}\circ h)$. Since $\tilde{f}$ is an approximation of $f$, we have that ${f(h(x))=d_\mathcal{Y}(h(x_0),h(x))\wedge 1<\tilde{f}(h(x))+\varepsilon}$ for all $x\in \mathcal{X}$. On the other hand, we have that $\tilde{f}(h(x))\leq \tilde{f}(h(x_0))+Ld_\mathcal{X}(x_0,x)\leq \varepsilon+Ld_\mathcal{X}(x_0,x)$. It follows that ${d_\mathcal{Y}(h(x_0),h(x))\wedge 1<Ld_\mathcal{X}(x_0,x)+2\varepsilon}$, which implies $\mathrm{SLip}~h(x_0)\leq L$ (which depends on $x_0)$. Smoothness of $h$ follows from the fact that the composition $\varphi \circ h$ is of class $C^1$ for any (non negative) compactly supported and $C^1$-smooth $\varphi$. We remark that the separation property (iii) holds here due to the fact that every semi-Lipschitz function on a Finsler manifold is continuous (in general, every semi-Lipschitz function is continuous for the topology of the symmetrized distance of its domain, which coincides with the manifold topology for any Finsler manifold, as shown in Chapter 6.2 C of \cite{BCS}). Case (d) is in fact a consequence of (e) (see also \cite{GJR}). Concerning case (g), the scalarization property is a consequence of compactness and Theorem 3.9 in \cite{GJ}, and the separation property follows from the comments in Remark \ref{flat}.

\end{proof}
We summarize the results of Corollary \ref{examplesbst} in the following table.
\begin{table}[h]
\begin{tabular}{|c|c|c|c|c|}
\hline
   $\mathcal{G}(X)$& Hypothesis on $X$ & Type of functions &  Algebraic structure& $\tau$ is a...\\
   \hline

     $C^{b,+}(X)$& - & continuous & linear & homeomorphism\\
      \hline

     $\mathrm{LIP}^\infty_+(X)$ & Quasiconvex& Lipschitz & linear & Lipschitz \\
     & & & & homeomorphism\\
      \hline

     $D^{b,+}(X)$ &-& bounded pointwise&  linear & $C_{SL}$\\
     & & Lipschitz constant& & homeomorphism\\
      \hline

     $D^{b,+}_{\mathrm{SL}}(X)$&-&bounded & cone& $C_{SL}$ \\
     & & metric slope & & homeomorphism\\
      \hline

    $C_{b,+}^1(\mathcal{X})$ &Reversible  &$C^1$, Lipschitz & linear & $C_{SL}$\\
    & Finsler manifold & & & diffeomorphism\\
     \hline

     $SC_{b,+}^1(\mathcal{X})$& Finsler manifold& $C^1$, semi-Lipschitz & cone & $C_{SL}$\\
     & & & & diffeomorphism\\
     \hline

     $\mathrm{lip}_+(X)$ &  compact, purely &little Lipschitz & linear & Lipschitz \\
     & 1-unrectifiable& & & homeomorphism\\
     \hline

\end{tabular}
\\
\caption{Summary table of Corollary \ref{examplesbst}}
\end{table}
\newpage
\subsection{Lipschitz version}

Several of the examples of Corollary~\ref{examplesbst} could be improved upon, for example, adding some form of quantitative control over the homeomorphism $\tau$. This approach makes sense when $\tau$ is, for instance, a Lipschitz homeomorphism, but not when $\tau$ is only a topological homeomorphism. In order to refine these results, we need to add stronger hypothesis that will yield stronger conclusions, at the expense of reducing the scope of the result.

\begin{theorem}\label{BSlip}
 Let $(X,d_X)$ and $(Y,d_Y)$ be complete metric spaces, let $\mathcal{F}(X,Y)$ be a subset of $C(X,Y)$, and let $\mathcal{G}(X)$ and $\mathcal{G}(Y)$ be subcones of $\mathrm{LIP}(X)\cap C^{b,+}(X)$ and $\mathrm{LIP}(Y)\cap C^{b,+}(Y)$, respectively, such that:
   \begin{itemize}
     \item[(i)] For $Z\in \{X,Y\}$, the cone $\mathcal{G}(Z)$ is endowed with a conic norm $\|\cdot|_Z$ which satisfies $\|\cdot|_Z\geq \max\{\mathrm{LIP}(\cdot),\|\cdot\|_\infty\}$, and which makes it a unital normed conic-semiring under the usual addition and multiplication of real-valued functions.
    \item[(ii)]  $h\in\mathcal{F}(X,Y)$ provided  $f\circ h\in \mathcal{G}(X)$ for all $f\in \mathcal{G}(Y)$.
    \item[(ii)'] $h\in\mathcal{F}(Y,X)$ provided  $f\circ h\in \mathcal{G}(Y)$ for all $f\in \mathcal{G}(X)$.
    \item[(iii)] For $Z\in \{X,Y\}$, $\mathcal{G}(Z)$ is uniformly separating for $(Z,d_Z)$, in the sense that, for every pair of subsets $A$ and $B$ of $Z$ with $d_Z(A,B)>0$, there exists some $f\in \mathcal{G}(Z)$ such that $\overline{f(A)}^{|\cdot|} \cap \overline{f(B)}^{|\cdot|} =\emptyset$.
    \item [(iv)] There exists a constant $C\geq 1$ such that, for $Z\in \{X,Y\}$ and for every pair of points $w,z\in Z$, there exists a function $f\in \mathcal{G}(Z)$ with $\|f|\leq C$ such that $f(z)-f(w)=d_Z(w,z)$.
\end{itemize}
\noindent For $Z\in\{X,Y\}$, denote $\mathcal{A}(Z)=\mathrm{span}(\mathcal{G}(Z))$, endowed with the extended asymmetric norm induced by $\mathcal{G}(Z)$ and its natural algebra structure. If $T:\mathcal{A}(Y)\to \mathcal{A}(X)$ is an order-preserving isomorphism of extended asymmetric normed algebras, there exists a bi-Lipschitz homeomorphism  $\tau\in \mathcal{F}(X,Y)$ with $\tau^{-1}\in \mathcal{F}(Y,X)$,  such that
$$Tf=f\circ \tau$$
for all $f\in \mathcal{A}(Y)$,   and satisfying that 
$$\mathrm{LIP}(\tau)\leq C\|T|\quad\text{  and  }\quad\mathrm{LIP}(\tau^{-1})\leq C\|T^{-1}|.$$ 
\end{theorem}
\begin{proof}
Clearly, all hypothesis for Theorem~\ref{BSt} are met. It only remains to prove the bound on the Lipschitz constant of $\tau$. Take two points $a,b\in X$, and let us estimate $d_Y(\tau (a),\tau (b))$. Condition $(iv)$ allows us to take $f\in \mathcal{G}(Y)$ with $\mathrm{LIP}(f)\leq C$ such that ${f(\tau(a))-f(\tau(b))=d_Y(\tau (a),\tau (b))}$. The composition formula yields that
$$d_Y(\tau (a),\tau (b))= Tf(a)-Tf(b).$$
Since the function $Tf$ is Lipschitz and $\mathrm{LIP}(Tf)\leq \|Tf|$ (by hypothesis $(i)$), we have that
$$
d_Y(\tau (a),\tau (b)) \leq \|Tf|d_X(a,b)\leq \|f|\|T|d_X(a,b)\leq C\|T|d_X(a,b),
$$
which implies $\tau$ is $C\|T|$-Lipschitz.

In the same way we obtain that $\tau^{-1}$ is also $C\|T^{-1}|$-Lipschitz.

\end{proof}
\begin{remark}\em
In what follows, the least constant $C$ satisfying condition $(iv)$ of Theorem~\ref{BSlip} will be called the \em{separation constant} of the families $\mathcal{G}(X)$ and $\mathcal{G}(Y)$.
\end{remark}

\begin{corollary}\label{examplesbslip}
For the following classes of spaces of real-valued functions $\mathcal{G}(X)$, we can obtain a quantitative bound on the homeomorphism $\tau:X\to Y$ obtained in terms of the separation constant $C$ and the norm of the isomorphism $T$. Let us denote $K=C\max\{\|T|, \|T^{-1}|\}$.
\begin{itemize}
    \item[(a)] $C_{b,+}^1(\mathcal{X})$ of bounded, non-negative functions with bounded derivative on a connected, complete and reversible Finsler manifold $\mathcal{X}$, endowed with the norm $\|f\|=\max\{\|f\|_\infty, \|df\|_\infty\}$, obtaining a Lipschitz diffeomorphism $\tau$ satisfying $$\max\{\|d\tau\|_\infty,\|d\tau^{-1}\|_\infty\}\leq K.$$
   % \item[$(b)$] $SC_b^1(X)$, of bounded semi-Lipschitz (or equivalently, of bounded functions with bounded derivative in the asymmetric sense) functions of class $C^1$ on a bicomplete Finsler manifold, obtaining a \textbf{semi}-Lipschitz diffeomorphism $\tau$ satisfying $$\max\{\|d\tau|_\infty,\|d\tau^{-1}|_\infty\}\leq K,$$
    %\noindent where the norm of the differential is computed in the asymmetric sense, see (JFA Finsler).
    \item[(b)] $\mathrm{LIP}_+^\infty(X)$, of bounded non-negative Lipschitz functions on a complete and quasi-convex metric space, endowed with the norm $\|f\|=\max\{\|f\|_\infty, \mathrm{LIP}(f)\}$, obtaining a Lipschitz homeomorphism $\tau$ satisfying $$\max\{\mathrm{LIP}(\tau),\mathrm{LIP}(\tau^{-1})\}\leq K.$$
    %\item[$(d)$] $\mathrm{SLIP}^\infty(X)$, of bounded semi-Lipschitz functions on a \textbf{quasi-metric} space, obtaining a \textbf{semi}-Lipschitz homeomorphism $\tau$ satisfying $$\max\{\mathrm{SLIP}(\tau),\mathrm{SLIP}(\tau^{-1})\}\leq K.$$
    \item[(c)] $\mathrm{lip}_+(X)$ of non-negative little Lipschitz functions on a compact and purely $1$-unrectifiable metric space $X$, endowed with the norm $\|f\|=\max\{\|f\|_\infty, \mathrm{LIP}(f)\}$, obtaining a Lipschitz homeomorphism $\tau$ satisfying $$\max\{\mathrm{LIP}(\tau),\mathrm{LIP}(\tau^{-1})\}\leq K.$$
    %\item[$(f)$] $\mathrm{slip}^\infty(X)$ of bounded and locally flat semi-Lipschitz function on a bicomplete and forward-compact \textbf{quasi}-metric space satisfying the separation property, obtaining a semi-Lipschitz homeomorphism $\tau$ satisfying $$\max\{\mathrm{SLIP}(\tau),\mathrm{SLIP}(\tau^{-1})\}\leq K.$$
\end{itemize}
\end{corollary}
In all the examples above the separation constant is $C=1$. In the Lipschitz case, this can be proved using distance functions. For the case of Riemannian and Finsler manifolds, this can be achieved by using smooth approximations of distance functions. For locally flat Lipschitz functions, it can be deduced from the fact that, for  boundedly compact metric spaces, the separation factor is always $1$ (see \cite[Corollary 4.40]{Weaver}). It follows that in all three cases of Corollary~\ref{examplesbslip}, we have $K=1$ whenever the isomorphism $T$ is in fact an isometry, which implies $\tau$ is also an isometry.

\begin{remark}\em
 In all cases mentioned in Corollary~\ref{examplesbslip}, the hypothesis of positivity of $T$ is unnecessary. Indeed, all algebras mentioned above are known to be closed under bounded inversion, which can be used to prove that every $\varphi$ in the structure space $\mathcal{S}(X)$ is positive, which guarantees that the dual operator $T^*$ sends $\mathcal{S}(X)$ into $\mathcal{S}(Y)$, thus eliminating the need for positivity of $T$. For the cases of Lipschitz functions, we can use that the algebras $\mathrm{LIP}_+^\infty(X)$ and $\mathrm{lip}_+(X)$ are also lattices, which allows us to use Lemma 2.3 of \cite{GJ} to prove that the elements of $\mathcal{S}(X)$ are also lattice homeomorphisms, and therefore, positive. In the manifold case, positivity of the elements of the structure space was shown in Section 6 of \cite{GJR}.
\end{remark}

\subsection{Pointwise Lipschitz version}\label{subsecbspointwise}

Our last result deals with pointwise Lipschitz functions and functions with bounded metric slopes. In order to obtain a bound on the pointwise Lipschitz constant of the desired homeomorphism between metric spaces $(X,d_X)$ and $(Y,d_Y)$, we will need the spaces to be uniformly locally radially quasi-convex (see Definition \ref{defuniflocallytadiallyqc}).

\begin{theorem}\label{BSpoint}
Let $(X,d_X)$ and $(Y,d_Y)$ be complete uniformly locally radially quasi-convex metric spaces, %let $\mathcal{F}(X,Y)$ be a subset of $C(X,Y)$,
and let $\mathcal{G}(X)$ and $\mathcal{G}(Y)$ be subcones of $D_{SL}^{b,+}(X)$ and $D_{SL}^{b,+}(Y)$, respectively, such that:
\begin{itemize}
     \item[(i)] For $Z\in\{X,Y\}$, the subcone $\mathcal{G}(Z)$ is endowed with a conic norm $\|\cdot|_Z$ which satisfies $\|\cdot|_Z\geq \max\{\|\mathrm{SLip}(\cdot)\|_\infty,\|\cdot\|_\infty\}$, and which makes it into a unital normed conic-semiring under the usual addition and multiplication of real-valued functions.
    \item[(ii)]  $h:X\to Y$ is pointwise Lipschitz provided $f\circ h\in \mathcal{G}(X)$ for all $f\in \mathcal{G}(Y)$.
    \item[(ii)'] $h:Y\to X$ is pointwise Lipschitz provided $f\circ h\in \mathcal{G}(Y)$ for all $f\in \mathcal{G}(X)$.
    \item[(iii)] For $Z\in \{X,Y\}$, $\mathcal{G}(Z)$ is uniformly separating for $(Z,d_Z)$, in the sense that, for every pair of subsets $A$ and $B$ of $Z$ with $d_Z(A,B)>0$, there exists some $f\in \mathcal{G}(Z)$ such that $\overline{f(A)}^{|\cdot|} \cap \overline{f(B)}^{|\cdot|} =\emptyset$.
    \item [(iv)] There exists a constant $C\geq 1$ such that, for $Z\in \{X,Y\}$ and for every pair of points $w,z\in Z$, there exists a function $f\in \mathcal{G}(Z)$ with $\|f|\leq C$ such that $f(z)-f(w)=d_Z(w,z)$.
\end{itemize}
\noindent For $Z\in\{X,Y\}$, denote $\mathcal{A}(Z)=\mathrm{span}(\mathcal{G}(Z))$, endowed with the extended asymmetric norm associated with $\mathcal{G}(Z)$, and its natural algebra structure. If $T:\mathcal{A}(Y)\to \mathcal{A}(X)$ is an order-preserving isomorphism of extended asymmetric normed algebras, there exists a bi-pointwise Lipschitz homeomorphism  $\tau:X \to Y$ such that
$$Tf=f\circ \tau$$
for all $f\in \mathcal{A}(Y)$. Moreover,
$$
\mathrm{Lip}(\tau)\leq C\, K_X \,\|T| \quad \text{  and  }\quad \mathrm{Lip}(\tau^{-1})\leq C\, K_Y \,\|T^{-1}|,
$$
where $K_X\geq 1$ and $K_Y\geq 1$ are the constants associated with the uniform local radial quasi-convexity of $X$ and $Y$, respectively.
\end{theorem}
\begin{proof}
Clearly, all hypothesis for Theorem~\ref{BSt} are met, if we consider  $\mathcal{F}(X,Y)$ to be the space of all pointwise Lipschitz functions from $X$ to $Y$. In this way we obtain a bi-pointwise Lipschitz homeomorphism $\tau:X \to Y$ such that $Tf=f\circ \tau$ for all $f\in \mathcal{A}(Y)$. It only remains to prove the bound on the pointwise Lipschitz constants of $\tau$ and $\tau^{-1}$. Fix a non isolated point $x_0\in X$, and let $U_{x_0}$ and $K_X$ be the neighborhood and constant given by the uniform local radial quasi-convexity of $X$. Then, for any point $x\in U_{x_0}$, let $\gamma_x:[a,b]\to U_{x_0}$ be a rectifiable curve such that $\gamma_x(a)=x_0$ and $\gamma_x(b)=x$. Using hypothesis (iv) of Theorem \ref{BSpoint}, take $f\in \mathcal{G}(Y)$ with $\|f|\leq C$ such that $f(\tau(x))-f(\tau(x_0))=d_Y(\tau (x_0),\tau (x))$. Using the composition formula, we get
\begin{equation*}\label{ineqpointw}
d_Y(\tau(x_0),\tau(x))=Tf(x_0)-Tf(x).%\leq C\|T| d_X(x_0,x).
\end{equation*}
Next, we apply Lemma \ref{lemmacurva} to the function $Tf$ and the curve $\gamma_x$ connecting $x_0$ and $x$, obtaining that
$$
d_Y(\tau(x_0),\tau(x))\leq \|\mathrm{SLip}(Tf)\|_\infty \ell(\gamma_x).
$$
Since $\|\mathrm{SLip}(Tf)\|_\infty\leq \|Tf|\leq \|T|\|f|\leq C\|T|$ and $\ell(\gamma_x)\leq K_X d_X(x_0,x)$, we conclude that
$$
 d_Y(\tau(x_0),\tau(x))\leq K_X C\|T|d_X(x_0,x) \text{ for any }x\in U_{x_0},
$$
which implies $\mathrm{Lip}(\tau)(x_0)\leq K_X C\|T|$.

Working with $T^{-1}$ we obtain he corresponding bound for $\mathrm{Lip}(\tau^{-1})$.
\end{proof}

\begin{corollary}\label{corbspoint}
    Theorem~\ref{BSpoint} can be applied to the following spaces, provided $X$ is a complete and uniformly locally radially quasi-convex metric space:
    \begin{itemize}
        \item[(a)] $\mathcal{G}(X)=D^{b,+}(X)$ of bounded non-negative functions with bounded pointwise Lipschitz constant on $X$.
        \item[(b)] $\mathcal{G}(X)=D_{SL}^{b,+}(X)$ of bounded non-negative continuous functions with bounded metric slope on $X$.
    \end{itemize}
In both cases, the homeomorphism $\tau$ is bi-pointwise Lipschitz, with
$$
\|\mathrm{Lip}(\tau)\|_\infty \leq K_X C\|T|,
$$
where $K_X\geq 1$ is the constant associated with he uniform local radial quasi-convexity of $X$.

\end{corollary}
\begin{remark}\em
In both cases mentioned in Corollary~\ref{corbspoint}, the hypothesis of positivity of $T$ is unnecessary. It was shown in \cite{DJ} that the algebra $D^\infty(X)$ is closed under bounded inversion, which is also known for $\mathrm{LIP}^\infty(X)$. Thus, it can be proven that every $\varphi$ in the structure space $\mathcal{S}(X)$ is positive, which guarantees that the dual operator $T^*$ maps $\mathcal{S}(X)$ into $\mathcal{S}(Y)$, thus eliminating the need for positivity of $T$. The same argument works for $D_{\mathrm{SL}}(X)$.
\end{remark}

\begin{corollary}\label{BSiso}
If the separation constant $C$ in hypothesis $(iv)$ of Theorem~\ref{BSlip} is $1$ and $T$ is an isometric isomorphism of extended asymmetric normed algebras, then $\tau $ is an isometry. If the separation constant $C$ in hypothesis $(iv)$ of Theorem~\ref{BSpoint} is $1$, as well as the constants $K_X$ and $K_Y$ associated with the uniform local radial quasi-convexity of $X$ and $Y$, respectively, and $T$ is an isometric isomorphism of extended asymmetric normed algebras, then $\tau $ is a pointwise isometry.
\end{corollary}

%%%%%%%%%%%%%%%%%%%%%%%%%%%%%%%%%%%%%

\end{document}